\documentclass[11pt,a4paper]{article}

\usepackage{cite}

\usepackage{amsmath}
\usepackage{amscd}
\usepackage{amssymb}
\usepackage{latexsym}
\usepackage{color}
\numberwithin{equation}{section}
\usepackage[normalem]{ulem}

\usepackage[colorlinks,
           linkcolor=blue,      
           anchorcolor=blue,  
           citecolor=red,       
            ]{hyperref}

\newcommand{\DD}{\mathbb{D}}

\renewcommand{\div}{{\rm div}\,}
\newcommand{\tr}{{\rm tr}\,}

  \newcommand{\vrp}{{\vr_+}}
    \newcommand{\vrm}{{\vr_-}}
      \newcommand{\np}{{n_+}}
    \newcommand{\nm}{{n_-}}
      \newcommand{\bvrp}{{\tilde\vr_+}}
    \newcommand{\bvrm}{{\tilde\vr_-}}
      \newcommand{\gp}{{\gamma^+}}
      \newcommand{\gm}{{\gamma^-}}
      \newcommand{\Hp}{{H_+}}
      \newcommand{\Hm}{{H_-}}
      \newcommand{\Pp}{{p_+}}
      \newcommand{\Pm}{{p_-}}
      \newcommand{\Nu}{\mathcal{V}}
\newcommand{\Nutx}{\mathcal{V}_{t,x}}
\newcommand{\lel}{\left\langle}
\newcommand{\ril}{\right \rangle}

\def\tilde{\widetilde}

\newcommand{\n}{\nabla}

\newcommand{\BR}{{\tilde{R}}}
\newcommand{\BQ}{{\tilde{Q}}}

\newcommand{\lap}{\Delta}

\newcommand{\Div}{\operatorname{div}}

\newcommand{\pf}{{\noindent\it Proof.~}}
\newcommand{\Ov}[1]{\overline{#1}}

\newcommand{\vS}{{\mathbb{S}}}
\newcommand{\vw}{\vc{w}}

\newcommand{\vr}{\varrho}

\newcommand{\vrn}{\vr_n}

\newcommand{\vun}{\vu_n}

\newcommand{\vu}{\vc{u}}
\newcommand{\vv}{\vc{v}}

\newcommand{\vc}[1]{{\bf #1}}
\newcommand{\vcg}[1]{{\pmb #1}}

\newcommand{\Grad}{\nabla}

\newcommand{\pt}{\partial_{t}}
\newcommand{\ptb}[1]{\partial_{t}(#1)}
\newcommand{\Dt}{\frac{ {\rm d}}{{\rm d}t}}

\newcommand{\dx}{{\rm d} {x}}

\newcommand{\dt}{{\rm d} t }

\newcommand{\dxdt}{\dx \,\dt}
\newcommand{\lr}[1]{\left( #1 \right)}
\newcommand{\intO}[1]{\int_{\Omega} #1 \ \dx}

\newcommand{\intOB}[1]{\int_{\Omega} \left( #1 \right) \ \dx}

\newcommand{\intTO}[1]{\int_0^T\!\!\!\! \int_{\Omega} #1 \ \dxdt}
\newcommand{\inttau}[1]{\int_0^\tau #1 \ \dt}
\newcommand{\inttauO}[1]{\int_0^\tau\!\!\!\! \int_{\Omega} #1 \ \dxdt}

\newcommand{\inttauOB}[1]{ \int_0^\tau\!\!\!\! \int_{\Omega} \left( #1 \right) \ \dxdt}

\newcommand{\eq}[1]{\begin{equation}
\begin{split}
#1
\end{split}
\end{equation}}
\newcommand{\eqh}[1]{\begin{equation*}
\begin{split}
#1
\end{split}
\end{equation*}}
\newcommand{\ep}{\varepsilon}

\newcommand{\R}{\mathbb{R}}
\newtheorem{thm}{Theorem}[section]
\newtheorem{lemma}[thm]{Lemma}
\newtheorem{prop}[thm]{Proposition}
\newtheorem{df}{Definition}
\newtheorem{rmk}[thm]{Remark}

\title{\bf{Dissipative measure-valued solutions and weak--strong uniqueness
for a viscous Baer--Nunziato system with pressure relaxation}}

\author{
 Nilasis Chaudhuri \thanks{Faculty of Mathematics, Informatics and Mechanics, University of Warsaw, Warsaw, Poland, E-mail: \texttt{nchaudhuri@mimuw.edu.pl}}\quad 
Milan Pokorn\'{y}\thanks{Faculty of Mathematics and
Physics Mathematical Institute, Charles University, Sokolovsk\'{a} 83,
186 75 Prague 8, Czech Republic, E-mail: \texttt{pokorny@karlin.mff.cuni.cz}} \quad 
Ewelina Zatorska\thanks{Mathematics Institute, University of Warwick, Zeeman Building, Coventry CV4 7AL, United Kingdom, E-mail:  \texttt{ewelina.zatorska@warwick.ac.uk}}
}

\begin{document}
 \maketitle 

\abstract
We study a viscous one-velocity Baer--Nunziato system for two barotropic
compressible fluids in a bounded three-dimensional domain. In contrast with
models in which the volume fraction is merely transported or determined by an
algebraic equilibrium constraint, it satisfies a differential closure driven
by the pressure gap between the phases. For arbitrary finite-energy initial
data and adiabatic exponents $\gamma^\pm>1$, we construct global-in-time
dissipative measure-valued solutions and prove dissipative
measure-valued--strong uniqueness relative to any sufficiently regular
solution with the same initial data. The principal difficulties are the singular and
non-continuous behavior of the pressure-relaxation source at vanishing volume
fractions, the associated concentration defects, and the lack of direct
coercivity of the standard thermodynamic relative energy with respect to the
volume fraction. These are resolved by an endpoint cutoff argument and an
augmented relative energy coupled to the renormalized volume-fraction
equation.

\medskip

\vspace{4mm}



\section{Introduction}\label{Sec:introduction}

Two-phase compressible flows combine the analytical difficulties of
compressible fluid mechanics with an additional geometric variable: the
volume occupied by each constituent. At the averaged level, the phase
fractions, phase densities, and velocities cannot be determined from the
balance laws alone, and a closure relation must be prescribed. The choice of
closure is not merely a modelling detail. It changes the mathematical
structure of the system, the available entropy, and the compactness mechanisms
that survive under approximation. Classical accounts of averaged multiphase
models are given by Drew--Passman \cite{DP} and Ishii--Hibiki \cite{IsHi}; see
also the hierarchy of reduced models in \cite{Bouchut} and the recent survey
\cite{Br-vulcano}. The modelling and computational relevance of the
Baer--Nunziato hierarchy is illustrated by the seven-equation formulation of
Saurel and Abgrall \cite{SaurelAbgrall99} and by mechanically equilibrated
reductions such as the models of Kapila et al.\ \cite{Kapila01} and
Murrone--Guillard \cite{MurroneGuillard05}.

The starting point of the present paper is a viscous one-velocity
Baer--Nunziato model. The original Baer--Nunziato theory \cite{BN} assigns
separate velocities and pressures to the phases. One-velocity reductions arise
after mechanical relaxation, but the pressure disequilibrium remains visible
through a source in the equation for the volume fraction. Such relaxation
terms have been derived from heterogeneous compressible flows in
\cite{BH19,BBR}; related pressure-relaxation limits and their stability have
been studied for strong solutions in \cite{Burtea,BCG26}. These derivations
provide an important reason to retain the volume-fraction equation rather than
replace it a priori by an algebraic pressure-equilibrium constraint, especially when the viscous models are considered.
More generally, entropy-dissipative relaxation is a classical theme in the
theory of hyperbolic balance laws; we refer to Chen--Levermore--Liu
\cite{CLL94} and to the relative-entropy framework of Tzavaras
\cite{Tzavaras05}.

A particularly close microscopic-to-macroscopic connection is provided by
Hillairet, Mathis, and Seguin \cite{HillairetMathisSeguin23}, who rigorously
derive a one-dimensional averaged bubbly-flow system, a
counterpart of the model studied here. For discussion of this result in the context of  spherical bubble models with moving sharp interfaces we refer to \cite{BurteaGerardVaret25,BurteaGerardVaret26}.

The system considered here reads
\begin{subequations}\label{intro-system}
\begin{align}
 \partial_tR+\Div(R\vu)&=0,\label{intro-R}\\
 \partial_tQ+\Div(Q\vu)&=0,\label{intro-Q}\\
 \partial_t((R+Q)\vu)+\Div((R+Q)\vu\otimes\vu)
 -\Div\vS(\Grad\vu)+\Grad p(\alpha,R,Q)&=\vc0,\label{intro-u}\\
 \partial_t\alpha+\vu\cdot\Grad\alpha
 &=\omega(\alpha,R,Q),\label{intro-alpha}
\end{align}
\end{subequations}
where $R=\alpha\vrp$ and $Q=(1-\alpha)\vrm$ denote the partial masses of the
two phases, $\alpha\in[0,1]$ is the volume fraction of phase~$+$,
and $\vu$ is the common velocity,
\begin{align}\label{intro-pressure-source}
 p(\alpha,R,Q)
 &:=\alpha p_+\left(\frac R\alpha\right)
 +(1-\alpha)p_-\left(\frac Q{1-\alpha}\right),\nonumber\\
 \omega(\alpha,R,Q)
 &:=\frac{\alpha(1-\alpha)}{\lambda+2\mu}
 \left[
 p_+\left(\frac R\alpha\right)
 -p_-\left(\frac Q{1-\alpha}\right)
 \right],
\end{align}
and $p_\pm(z)=z^{\gamma^\pm}$ with $\gamma^\pm>1$. 
We consider \eqref{intro-system} in a bounded domain
$\Omega\subset\mathbb R^3$, with no-slip condition on $\vu$. Following \cite{BBR,Burtea}, we consider the same Newtonian stress tensor $\vS=\vS(\Grad\vu)$  in each of the two phases:
\eq{\label{stres:mu_lambda}
\vS(\Grad\vu)=\mu\lr{\Grad\vu+\Grad^T\vu}+\lambda\Div\vu\mathbb{I},}
where  $\mu,\lambda$ are two viscosity coefficients.

The source $\omega$ drives the two pressures towards equilibrium whenever both
phases are present. Its mobility $\alpha(1-\alpha)$ vanishes at
$\alpha=0,1$, where the relaxation mechanism degenerates.

 With
\[
 H_\pm(z)=\frac{z^{\gamma^\pm}}{\gamma^\pm-1},
\qquad
 P(\alpha,R,Q)
 =\alpha H_+\left(\frac R\alpha\right)
 +(1-\alpha)H_-\left(\frac Q{1-\alpha}\right),
\]
the total energy density is given by
\[
 E(\alpha,R,Q,\vu)=\frac12(R+Q)|\vu|^2+P(\alpha,R,Q).
\]
The formal energy balance reveals both the stabilizing pressure-relaxation dissipation
\begin{equation}\label{intro-relaxation-dissipation}
 \frac{\alpha(1-\alpha)}{\lambda+2\mu}
 \left[
 p_+\left(\frac R\alpha\right)
 -p_-\left(\frac Q{1-\alpha}\right)
 \right]^2.
\end{equation}
At $\alpha=0$ or $\alpha=1$, however, the ratios in
\eqref{intro-pressure-source} are singular. The energy gives only weak
compactness of the phase variables, provides no spatial regularization of
$\alpha$, and does not keep the volume fraction uniformly away from its
endpoints. When suitably extended, the energy $E$ is  a lower-semicontinuous function, thus finite energy
forces $R=0$ on $\{\alpha=0\}$ and $Q=0$ on $\{\alpha=1\}$, but neither the
pressure-relaxation source nor its square is a continuous function on the closed phase space.
Identifying the transport term, pressure, and relaxation source in an
approximation limit therefore requires a Young-measure formulation together
with a separate analysis of the endpoint regions. This endpoint phenomenon is
one of the principal analytical difficulties of the paper.

The purpose of this work is to construct global dissipative measure-valued
solutions to \eqref{intro-system} and to prove their weak--strong uniqueness.


\medskip
\noindent\textbf{Relation to existing theories.}
It is useful to distinguish the present differential pressure-relaxation
closure from three nearby classes of models.

First, the one-dimensional one-velocity liquid--gas drift-flux systems of
Evje and coauthors \cite{EvjeKarlsen08,EvjeWenZhu17} also consist of two
partial-mass equations coupled to one mixture momentum balance, and the later
theory allows transitions to single-phase states. 
 Their pressure is prescribed directly in terms of the
partial masses, however, and there is no independent volume-fraction equation.

Second, in the \emph{pure-transport closure}, the volume fraction  is transported without the pressure-gap
source in \eqref{intro-alpha}. Thus, the transport structure is
fundamentally simpler in the relative-energy argument: no nonlinear
pressure-difference source has to be compared between the weak and strong
states. 
Global finite-energy weak solutions for related
one-velocity bifluid systems were constructed in \cite{Novotny19,KKNN};
weak--strong uniqueness and dissipative turbulent solutions were studied in
\cite{JN19,JKNN}, and the heat-conducting case in \cite{KN24}. 

Third, in the \emph{algebraic closure}, the phase pressures are constrained
by an implicit equilibrium relation, from which the volume fraction equation can be
recovered. In \eqref{intro-system}, by contrast, $\alpha$ is an
independent unknown governed by a nonlinear differential equation.
Finite-energy solutions for algebraic closure systems were obtained in
\cite{BMZ19,NM20}, with some important first contributions for simpler systems developed in \cite{VaWeYu, FKNZ}, see also \cite{Wen} for transitions to single-phase states
\cite{Wen}; local strong solutions and weak--strong uniqueness were studied in
\cite{PZ1,LLPZ}.


While the Lions--Feireisl compactness theory for compressible
Navier--Stokes equations \cite{Lions,Fe2002} 
provides a natural benchmark for the multi-fluid models, it does not apply directly here. As the pressure depends on $(\alpha,R,Q)$ and the relaxation source
degenerates at the boundary of the phase space, the usual effective-viscous-flux route does not directly yield the compactness needed for a global weak solution.
More flexible compactness
methods cover certain nonstandard pressures, stresses, and source terms
\cite{BrJa,VZ17}, but do not directly control the simultaneous oscillations of
the three phase variables and the endpoint degeneration encountered here.

This motivates a dissipative measure-valued formulation. The underlying
ideas originate in the measure-valued theory of DiPerna \cite{DiPerna85} and
the relative-entropy method of Dafermos \cite{Dafermos79}; general
weak--strong uniqueness results include \cite{BDLS11,DST12}. For compressible
Navier--Stokes equations, the framework was developed in
\cite{FG2W,AFN21,C2020}. It records oscillations by a parametrized Young
measure and concentrations by energy-controlled defects, while weak--strong
uniqueness forces both to vanish on the lifespan of a regular solution.


\medskip
\noindent\textbf{Main results and novelty.}
The paper establishes a global existence and stability theory for
\eqref{intro-system}. To the best of our knowledge, this is the first such
theory for the viscous barotropic one-velocity Baer--Nunziato system with the
pressure-relaxation closure \eqref{intro-alpha}.

Our first result is the global-in-time existence of dissipative
measure-valued solutions for arbitrary finite-energy initial Young measures
and the full range $\gamma^\pm>1$. The Young measure acts simultaneously on
the volume fraction, the partial masses, the velocity, and its symmetric
gradient; the formulation retains the renormalized volume-fraction equation,
and the energy defect controls the concentration measure in the momentum
balance.

The construction uses a parabolic regularization of the three scalar
equations and a Faedo--Galerkin approximation of the momentum equation. The
limits are deliberately separated. For fixed Galerkin dimension, the
finite-dimensional regularity of the velocity and maximum-principle estimates
give strong compactness as the artificial diffusion tends to zero. Only then
is the Galerkin dimension sent to infinity. The separation of the volume
fraction from $0$ and $1$ is available at the first limit but is not uniform
in the Galerkin dimension. To identify the relaxation source in the final
Young-measure limit, we exploit the exact relation
\[
 |\omega(\alpha,R,Q)|^2
 =
 \frac{\alpha(1-\alpha)}{\lambda+2\mu}\,
 \mathcal R(\alpha,R,Q),
\]
where $\mathcal R$ denotes the quantity in
\eqref{intro-relaxation-dissipation}. A cutoff supported away from the
endpoints makes the source continuous on the phase space, while the
relaxation dissipation shows that the contribution of the endpoint layers
vanishes uniformly. This identifies the source without imposing a uniform
lower bound on $\alpha$ or $1-\alpha$. A similar endpoint analysis is required to prove positivity of the concentration defects in the limiting energy inequality.

Our second main result is dissipative measure-valued--strong uniqueness.
First we observe that the existence of local-in-time strong-solution in
maximal-$L^p$--$L^q$ spaces can be obtained by adapting the strategy of \cite{PZ1}. Apart from regularity, these solutions are characterized by
positivity of the partial masses and strict separation of the strong volume
fraction from the endpoints.
We use them as comparison class in the relative energy functional. The
standard thermodynamic part of the functional controls the phase-density
perturbations but does not, by itself, provide the control of the
independent volume fraction needed for \eqref{intro-alpha}. We therefore
augment it by
\[
 \frac12|\alpha-\beta|^2,
\]
where $\beta$ is the strong volume fraction. This requires to work with the renormalized identity for the
volume-fraction equation and directly impacts the order of limit passages in our approximation procedure.
The augmented relative energy and the concentration defect are ultimately controlled by  Gronwall's inequality. Consequently, a dissipative
measure-valued solution emanating from strong initial data is a Dirac mass
concentrated on the strong solution throughout its lifespan.

Besides uniqueness, this stability result has at least two useful consequences.
Firstly, it may supply a convergence criterion for consistent
energy-stable approximations that generate dissipative measure-valued
solutions, in the spirit of the numerical measure-valued framework
\cite{FMT2016,FLM18,FLMS19,FLMSBook}. The additional control of the independent
volume fraction may also be useful in low-Mach limits with differential
closure; related compressible and algebraically closed two-fluid limits are
studied in \cite{FKNZ,LebotSurvey,Lebot26,LLZ26}.

\medskip
\noindent\textbf{Organization of the paper.}
In Section~\ref{Sec:2} we formulate the pressure-relaxation system and its energy
structure. Section~3 introduces dissipative measure-valued solutions, states
the global existence and weak--strong uniqueness theorems, and gives the
local strong-solution result. In Section~4 we construct the measure-valued
solution through the separated approximation limits. In
Section~\ref{Sec:MV-WS} we derive the relative-energy inequality, including
the additional identity for the volume fraction. Section~\ref{Sec:5} closes
the stability estimates and proves weak--strong uniqueness.

\section{Formulation of the problem}\label{Sec:2}

Recall that with the partial masses $R=\alpha\vrp$ and $Q=(1-\alpha)\vrm$ 
introduced in Section~\ref{Sec:introduction}, the system takes the form
\begin{subequations}\label{S}
\begin{align}
  &\pt R+ \Div(R\vu )=0,\label{SR}\\ 
  &\pt Q + \Div(Q\vu )=0 ,\label{ST}\\
&\ptb{(R+Q)\vu}+\Div((R+Q)\vu\otimes\vu)-\Div \vS(\Grad\vu) \nonumber\\
&\qquad\qquad\qquad  +\Grad \lr{\alpha p_{+}\lr{\frac{R}{\alpha}}+(1-\alpha)p_{-}\lr{\frac{Q}{1-\alpha}}}=\vc{0},\label{Su}\\
&\pt\alpha+\vu\cdot\Grad\alpha=\frac{\alpha(1-\alpha)}{\lambda+2\mu}\lr{p_{+}\lr{\frac{R}{\alpha}}-p_{-}\lr{\frac{Q}{1-\alpha}}},\label{Salpha}
\end{align}
\end{subequations} 
where the new unknowns are $R,Q,\vu,$ and $\alpha$.

We supplement the system  \eqref{S} with the following set of initial conditions:
\begin{equation} \label{eq1.3bi}
\begin{aligned}
R(0,x) &= \alpha _0\vr_{+,0}(x)=:R_0(x), \\
Q(0,x) &= (1-\alpha _0) \vr_{-,0}(x)=:Q_0(x), \\
(R+Q)\vu(0,x)&= 
(\alpha _0\vr_{+,0}+(1-\alpha _0)\vr_{-,0})\vu_0(x)=:\vc m_0,\\
\alpha(0,x)&=\alpha_0(x).
\end{aligned}
\end{equation}

The system is endowed  with the zero Dirichlet boundary condition
\begin{equation} \label{eq1.2bi}
\vu(t,x)\Big|_{\partial\Omega} = \vc{0},
\end{equation}
where $\Omega$ is a sufficiently smooth bounded domain of $\R^3$ and $T>0$ is arbitrary large. We will further denote $I=(0,T)$.

For convenience, we consider the Newtonian stress tensor in the following form
\eq{
    \vS(\nabla \vu):= &\mu \lr{{\nabla \vu+ \nabla^{T} \vu}-\frac{2}{3} \Div \vu \; \mathbb{I} } + \xi \div \vu \mathbb{I},
}
 where $\xi$ is related to the Lam\'e coefficient $\lambda$ through
 $2\mu+3\lambda=3\xi$.
 For any $\mathbb{M} \in \R^{3\times 3}$, we will denote 
  \[  \mathbb{T} (\mathbb{M}) = \frac{\mathbb{M}+ \mathbb{M}^{T}}{2}- \frac{1}{3}\tr \mathbb{M}\,\mathbb I.\]
 In particular, for $\mathbb{M}=\Grad\vu$, we have 
  \begin{align*}
     \mathbb{T} (\nabla \vu ) = \mathbb{D}\vu- \frac{1}{3} \tr \mathbb{D}\vu \mathbb{I}, \quad \text{where}\quad   \mathbb{D}\vu  = \frac{ \nabla \vu + \nabla^{T} \vu }{2}. 
  \end{align*}
  Therefore $ \mathbb{T} (\nabla \vu )$ depends only on the symmetric part of the gradient $\mathbb{D}\vu$ and so does $\vS$. With slight abuse of notation we will write
  $$\mathbb{T} (\nabla \vu )= \mathbb{T} (\mathbb{D}\vu ),\quad \vS(\Grad\vu)=\vS(\mathbb{D}\vu).$$
  Moreover, we have
\eqh{
\vS(\mathbb{D}\vu):\nabla\vu&= \vS(\mathbb{D}\vu):\mathbb{D}\vu\\
&=
\lr{2\mu \mathbb{T} (\mathbb{D}\vu) + \xi \tr \mathbb{D}\vu \mathbb{I}}:\mathbb{D}\vu
= 2\mu\,\bigl|\mathbb{T}(\mathbb{D}\vu )\bigr|^2
+\xi\,\bigl|\tr \mathbb{D}\vu \bigr|^2.
}
We will therefore make the following assumption on the viscosity coefficients 
\eqh{\mu>0\quad \text{and}  \quad \xi \geq 0.}
Equivalently, in the Lam\'e convention \eqref{stres:mu_lambda},
$\mu>0$ and $\lambda+\frac23\mu\geq0$. Notice also that
$\lambda+2\mu=\xi+\frac43\mu>0$, so the coefficient in the relaxation law is
well defined. The strengthened conformal Korn--Poincar\'e compatibility
condition below controls the trace fluctuation even when $\xi=0$; see
Remark~\ref{rmk:korn-choice}.


\section{Definitions and main results}
In this section we first carefully formulate the definitions of the measure-valued and strong solutions, and then we formulate our main results. 

\subsection{Dissipative measure-valued solutions solutions}
We consider  the following phase space for measure-valued solutions:
\eq{\label{phase}
\mathcal{F}= \left\{(s, r, q,\vv, \mathbb{D}_{\vv})  \mid s \in [0,1],\ r\in [0,\infty),\ q\in [0,\infty),\ \vv\in  \R^3,\ \mathbb{D}_{\vv} \in\R^{3\times 3}_{\text{sym}} \right\}.
}
\begin{df}\label{MVdef}
		We say that a parametrized measure $\{ \Nu_{t,x} \}_{(t,x)\in I\times \Omega}$,
	\begin{align*}
		\Nu \in L^{\infty}_{\text{weak-(*)}} \big( I\times \Omega ;\mathcal{P}(\mathcal{F} )\big),
	\end{align*}
	on the phase space $\mathcal{F}$ defined in \eqref{phase}
	is a measure--valued solution of the two-fluid  system \eqref{S} in $I\times \Omega$, with the boundary conditions \eqref{eq1.2bi} and the initial conditions \eqref{eq1.3bi} and with the dissipation defect $\mathcal{D}$,
	\begin{align*}
		\mathcal{D}\in L^{\infty}(0,T),\; \mathcal{D}\geq 0,
	\end{align*}
	if the following conditions hold:
	\begin{itemize}
		\item The unknowns enjoy the following regularity: 
		\begin{align}
			&  \alpha=\lel \Nu_{t,x}; s  \ril  \in L^\infty(I\times\Omega)\cap C_{weak}(I,L^2(\Omega)),\label{reg1}\\
			& 	R=\lel \Nu_{t,x}; r  \ril  \in C_{\text{weak}}(0,T;L^{\gamma^+}(\Omega)) ,\label{reg2r}\\
            & 	Q=\lel \Nu_{t,x}; q  \ril  \in C_{\text{weak}}(0,T;L^{\gamma^-}(\Omega)) ,\label{reg2q}\\
			&\vu=\lel \Nu_{t,x};  \vv \ril \in L^2(0,T;W^{1,2}(\Omega)).\label{reg4} 
		\end{align}	
In addition, the velocity and gradient variables are compatible with the no--slip
		boundary condition:  for every 
		$\mathbb M\in C^1((0,T)\times\Ov{\Omega};\R^{d\times d}_{\rm{sym}})$,
		\begin{align}\label{mv-gradient-compatibility}
		\intTO{ \lel\Nu_{t,x};\vv\ril\cdot
		\Div \mathbb M}
		=-\intTO{\lel\Nu_{t,x};\mathbb{D}_{\vv}\ril:
		\mathbb M}.
		\end{align}
		\item For a.e. $\tau \in (0,T) $ and $\psi \in C^{1}([0,T]\times {\Omega})$ we have
		\begin{align} \label{mv eqn 1r}
			\begin{split}
				&\int_{\Omega}R(\tau, \cdot) \psi(\tau, \cdot) \dx - \int_{ \Omega} R_0 \psi(0, \cdot) \dx \\
				&\qquad = \inttauO{ \left[ R  \partial_{t}\psi + \lel \Nu_{t,x};  r\vv\ril \cdot \Grad \psi \right] }.
			\end{split}
		\end{align}
        	\item For a.e. $\tau \in (0,T) $ and $\psi \in C^{1}([0,T]\times {\Omega})$ we have
		\begin{align} \label{mv eqn 1q}
			\begin{split}
				&\int_{\Omega}Q(\tau, \cdot) \psi(\tau, \cdot) \dx - \int_{ \Omega} Q_0 \psi(0, \cdot) \dx \\
				&\qquad = \inttauO{\left[ Q  \partial_{t}\psi + \lel \Nu_{t,x};  q\vv\ril \cdot \Grad \psi \right] }.
			\end{split}
		\end{align}
		\item There exists a concentration measure $r^{M} \in L^\infty_{\text{weak-(*)}}(0,T;\mathcal{M}(\overline{\Omega};\R^{d\times d})) + \mathcal{M}([0,T]\times \overline{\Omega};\R^{d\times d})$ and $\chi\in L^{1}(0,T)$ such that for a.e. $\tau \in (0,T)$, for all $\pmb{\phi}\in C^{1}([0,T]\times \overline{\Omega};\R^d)$ we have	
		\begin{align}\label{mv eqn 2}
			\begin{split}
				&\int_{ \Omega}  \lel\Nu_{\tau,x}; (r+q)\vv \ril \cdot \pmb{\phi}(\tau,\cdot) \dx - \int_{ \Omega} \lel \Nu_{0} ; (r+q)\vv  \ril \cdot \pmb{\phi}(0,\cdot) \dx\\
				&= \inttauO{\left[ \lel \Nu_{t,x};  (r+q)\vv  \ril \cdot \partial_{t} \pmb{\phi} +  \lel \Nu_{t,x};(r+q)\vv\otimes  \vv \ril : \Grad \pmb{\phi}+ \lel \Nu_{t,x}; p(s,r,q) \ril  \Div\pmb{\phi} \right] }\\
				& \qquad- \inttauO{\lel \Nu_{t,x}; \vS(\mathbb{D}_{\vv})\ril:  \Grad \pmb{\phi}}+ \lel r^M;\Grad \pmb{\phi}\ril_{\mathcal{M}([0,\tau]\times \overline{\Omega}), C([0,\tau]\times \overline{\Omega}),} 
			\end{split}
		\end{align}
moreover the following compatibility condition is satisfied.
		\begin{equation}\label{mom-def}
		\left\vert \lel r^M;\Grad \pmb{\phi}\ril_{\mathcal{M}([0,\tau]\times \overline{\Omega}), C([0,\tau]\times \overline{\Omega})} \right\vert \leq  \int_0^\tau \chi(t) \mathcal{D}(t) \Vert \pmb{\phi}(t,\cdot) \Vert_{C^1(\overline{\Omega})} \dt.
	\end{equation} 
    	\item For a.e. $\tau \in (0,T) $, for every
        $b\in C^1([0,1])$ and $\psi \in C^{1}([0,T]\times \overline{\Omega})$ we have
        \begin{align}\label{mv eqn 3}
        &\int_\Omega\lel\Nu_{\tau,x};b(s)\ril\psi(\tau)\,\dx
        -\int_\Omega\lel\Nu_{0,x};b(s)\ril\psi(0)\,\dx\nonumber\\
        &=\inttauO{\lel\Nu_{t,x};b(s)\ril\partial_t\psi
        +\lel\Nu_{t,x};b(s)\vv\ril\cdot\Grad\psi}\nonumber\\
        &\quad+\inttauO{\lel\Nu_{t,x};b(s) \,\tr \mathbb{D}_{\vv}
        +b'(s)\frac{s(1-s)}{\lambda+2\mu}
        \big(p_+(r/s)-p_-(q/(1-s))\big)\ril\psi}.
        \end{align}


		 \item 
        For a.e. $ \tau \in (0,T) $ we have
		\begin{align*}
			\begin{split}
				&\int_{ \Omega} \lel \Nu_{\tau ,x}; \frac{1}{2}  (r+q)|\vv|^2 + P(s,r,q) \ril \dx\\ 
&\qquad+\frac{1}{\lambda+2\mu}\inttauO{\lel \Nu_{t,x}; s(1-s)(p_{+}(r/s) -p_{-}(q/(1-s)))^2\ril }\\
       &\qquad+\inttauO{\lel \Nu_{t,x};\vS(\mathbb{D}_{\vv}): \mathbb{D}_{\vv}\ril}+ \mathcal{D}(\tau) \\
				&\leq \int_{ \Omega} \lel \Nu_{0,x};  \frac{1}{2}  (r+q)|\vv|^2 + P(s,r,q) \ril \dx, 
			\end{split}
		\end{align*}
        where we denoted
        \eq{\label{Psrq}
        P(s,r,q):=s \Hp (\np)+(1-s)\Hm(\nm),}
        \eq{\label{npnm}
        \np:=\frac{r}{s},\quad \nm:=\frac{q}{(1-s)},}
        and
\eq{\label{HpHm}
\Hp(\np):=\frac{1}{\gp-1}\np^\gp,\quad \Hm(\nm):=\frac{1}{\gm-1}\nm^\gm.
}
        
\item Moreover, for a.e. $\tau\in(0,T)$ the generalized trace-free
		Korn--Poincar\'e inequality
		\begin{align}\label{mv-poincare-korn}
		&\inttauO{
		\lel\Nu_{t,x};|\vv-\vw|^2
        +|\mathbb D_{\vv}-\mathbb D\vw|^2\ril}\leq C_{\rm PK}\inttauO{
		\lel\Nu_{t,x};|\mathbb{T}(\mathbb{D}_{\vv})-
		\mathbb{T}(\Grad \vw)|^2\ril},
		\end{align}
   holds for all $\vw\in L^2(0,T; W^{1,2}_0(\Omega;\R^3))$.
	\end{itemize}
	\end{df}
\begin{rmk}[Choice of the Korn--Poincar\'e compatibility condition]
\label{rmk:korn-choice}
We use the strengthened trace-free formulation \eqref{mv-poincare-korn}.
For zero-boundary vector fields the underlying deterministic estimate is the
$L^p$ conformal Korn inequality
\begin{align}\label{conformal-korn}
 \|\mathbf z\|_{W^{1,p}(\Omega)}
 \leq C_p\|\mathbb T(\nabla\mathbf z)\|_{L^p(\Omega)},
 \qquad \mathbf z\in W^{1,p}_0(\Omega;\R^3),\quad 1<p<\infty;
\end{align}
see \cite{Dain06,LewintanNeff21}. Lemma~2.2 of \cite{BFN} states the
Young-measure version with the velocity term on the left. The same proof,
applied to the full $W^{1,p}$ estimate \eqref{conformal-korn}, gives the
additional symmetric-gradient term in \eqref{mv-poincare-korn}, i.e. 
\[  \|\mathbf z\|_{L^{p}(\Omega)} + \|\mathbb{D} \mathbf z\|_{L^{p}(\Omega)} \leq  C_p\|\mathbb T(\nabla\mathbf z)\|_{L^p(\Omega)} .\]
This is the
formulation naturally controlled by the shear dissipation
$2\mu|\mathbb T(\mathbb D_{\vv})|^2$ and, unlike an estimate with the full
symmetric gradient on the right, it also controls the trace fluctuation when
$\xi=0$. Thus no strict positivity of the bulk viscosity is needed.
\end{rmk}

Our first main theorem concerns the global-in-time existence of measure-valued solutions
\begin{thm} \label{t:mv}
Assume for some
$\Nu_{0} \in L^{\infty}_{\text{weak-(*)}} \big( \Omega ;\mathcal{P}(\mathcal{F} )\big) $ that 
\eq{\label{MV:IE:1}
     &\int_{ \Omega} \lel \Nu_{0,x};  \frac{1}{2}  (r+q)|\vv|^2 + P(s,r,q) \ril \dx\leq C
}
and
\eq{\label{MV:IE:1a}\alpha_0 := \int_{ \Omega} \lel \Nu_{0,x};s\ril\dx \in [0,1].}
Then there exists a paramerized measure $\{ \Nu_{t,x} \}_{(t,x)\in I\times \Omega}$ and a dissipation defect ${\mathcal{D}}$ such that the couple $(\Nu_{t,x}, \mathcal{D})$ is a measure-valued solution to system \eqref{S} in the sense of Definition \ref{MVdef}.
\end{thm}

\subsection{Local-in-time strong solutions}

Based on the result in \cite{PZ1} we can prove the following

\begin{thm} \label{t:local}
Let $\Omega\subset\R^3$ be a bounded uniform $C^2$ domain and let
\begin{align*}
 2<p<\infty,\qquad 3<q<\infty,\qquad
 \frac{2}{p}+\frac{3}{q}<1.
\end{align*}
Assume that
\begin{align*}
 &(R_0,Q_0,\alpha_0)\in W^{1,q}(\Omega)^3,\qquad
 \vu_0\in B^{2-2/p}_{q,p}(\Omega;\R^3),
 \qquad \vu_0|_{\partial\Omega}=\vc 0,
\end{align*}
and that, for some $\underline r>0$, $\underline a\in(0,1/2)$ and
$L>0$,
\begin{align}
 &R_0\geq \underline r,\qquad Q_0\geq \underline r,
 \qquad \underline a\leq\alpha_0\leq1-\underline a
 \quad\text{in }\overline\Omega,                                      \label{strong-data-bounds}\\
 &\|R_0\|_{W^{1,q}(\Omega)}+\|Q_0\|_{W^{1,q}(\Omega)}+\|\alpha_0\|_{W^{1,q}(\Omega)}
 +\|\vu_0\|_{B^{2-2/p}_{q,p}(\Omega)}\leq L.                                 \label{strong-data-norm}
\end{align}
Then there exists $T=T(L,\underline r,\underline a)>0$ such that
\eqref{S}, supplemented with \eqref{eq1.2bi} and \eqref{eq1.3bi},
admits a unique strong solution satisfying
\begin{align}
 \vu&\in W^{1,p}(0,T;L^q(\Omega;\R^3))
 \cap L^p(0,T;W^{2,q}(\Omega;\R^3)\cap W^{1,q}_0(\Omega;\R^3)),
                                                                    \label{strong-u-reg}\\
 (R,Q,\alpha)&\in W^{1,p}(0,T;W^{1,q}(\Omega)^3).                  \label{strong-rqa-reg}
\end{align}
Moreover, $T$ may be chosen so that
\begin{align}
 R,Q\geq\frac{\underline r}{2},\qquad
 \frac{\underline a}{2}\leq\alpha\leq1-\frac{\underline a}{2},
 \qquad
 \int_0^T\|\nabla\vu(t)\|_{L^\infty(\Omega)}\,\dt\leq\delta,     \label{strong-positive}
\end{align}
where $\delta>0$ is any fixed sufficiently small constant.  The norm of the
solution in the spaces \eqref{strong-u-reg}--\eqref{strong-rqa-reg} is bounded
by a constant depending only on $L$, $\underline r$, $\underline a$, the
domain, the viscosity coefficients and the pressure laws.
\end{thm}

\begin{rmk}[Lagrangian proof strategy]\label{rmk:strong-proof}
Theorem \ref{t:local} is obtained by adapting the Lagrangian
$L^p$--$L^q$ maximal-regularity argument from \cite{PZ1}; it is not a direct
consequence of the theorem proved there, because the closure in \cite{PZ1} is
algebraic, whereas \eqref{Salpha} is an additional evolution equation.  We
indicate the new step.  The pressure-relaxation structure is also studied in
\cite{Burtea}, but there the well-posedness theory is developed in the whole
space near an equilibrium and does not directly cover the present
initial-boundary value problem.

For a trial Lagrangian velocity $\vv$, set
\begin{align*}
 X_{\vv}(t,y)=y+\int_0^t\vv(s,y)\,\mathrm ds,\qquad
 F_{\vv}=\nabla_yX_{\vv},\qquad A_{\vv}=F_{\vv}^{-1},
 \qquad J_{\vv}=\det F_{\vv}.
\end{align*}
If $r=R\circ X_{\vv}$, $s=Q\circ X_{\vv}$ and
$a=\alpha\circ X_{\vv}$, the two continuity equations give
\begin{align}
 r=\frac{R_0}{J_{\vv}},\qquad s=\frac{Q_0}{J_{\vv}},                \label{lagrangian-rs}
\end{align}
while the PDE closure becomes the pointwise ODE
\begin{align}
 \partial_ta
 =\frac{a(1-a)}{\lambda+2\mu}
 \left[
 p_+\left(\frac{R_0}{aJ_{\vv}}\right)
 -p_-\left(\frac{Q_0}{(1-a)J_{\vv}}\right)
 \right],\qquad a|_{t=0}=\alpha_0.                                 \label{lagrangian-alpha}
\end{align}
On the compact state set determined by \eqref{strong-data-bounds}, the
right-hand side of \eqref{lagrangian-alpha} is a $C^2$ Nemytskii map on the
Banach algebra $W^{1,q}(\Omega)$.  Differentiating
\eqref{lagrangian-alpha} in $y$, followed by Gronwall's inequality, gives
$a\in W^{1,p}(0,T;W^{1,q})$ and the Lipschitz dependence of $a$ on $\vv$.
This is the additional lemma needed for the PDE closure.

Writing
\begin{align*}
 \Pi(r,s,a)=a p_+\left(\frac r a\right)
 +(1-a)p_-\left(\frac{s}{1-a}\right),\qquad m_0=R_0+Q_0,
\end{align*}
the Lagrangian momentum equation is
\begin{align}
 &m_0\partial_t\vv
 -\Div_y\left[J_{\vv}\vS(\nabla_y\vv A_{\vv})A_{\vv}^{T}\right]\nonumber\\
 &\hspace{25mm}
 +J_{\vv}A_{\vv}^{T}\nabla_y\Pi(r,s,a)=0.                         \label{lagrangian-momentum}
\end{align}
Its principal part at $A_{\vv}=\mathbb I$, $J_{\vv}=1$ is the
weighted Lam\'e problem
\begin{align*}
 m_0\partial_t\vv-\mu\Delta\vv-(\mu+\lambda)\nabla\Div\vv=\vc f,
 \qquad \vv|_{\partial\Omega}=\vc0.
\end{align*}
The required maximal regularity follows from the linear theory in
\cite{EnomotoShibata,PSZmax}; the fixed-point organization and nonlinear
estimates follow \cite{PZ1}.  Indeed,
\begin{align*}
 \|A_{\vv}-\mathbb I\|_{L^\infty(0,T;W^{1,q}(\Omega))}
 +\|J_{\vv}-1\|_{L^\infty(0,T;W^{1,q}(\Omega))}
 \leq C T^{1/p'}\|\vv\|_{L^p(0,T;W^{2,q}(\Omega))}.
\end{align*}
Together with the Lipschitz estimate for \eqref{lagrangian-alpha}, this
produces a factor $T^\theta$, $\theta>0$, in the difference of the nonlinear
right-hand sides of \eqref{lagrangian-momentum}.  Banach's fixed-point
theorem then proves local existence, uniqueness and continuous dependence.
Notice that the full norms in \eqref{strong-data-norm}, rather than only the
norms of the gradients, are needed to control the pressure and the lifespan.
\end{rmk}


\begin{rmk}[Regularity used in the relative-energy argument]\label{r:stronger_reg}
The regularity from \eqref{strong-u-reg} implies
$\vu\in L^1(0,T;W^{1,\infty}(\Omega))$, since
$W^{2,q}\hookrightarrow W^{1,\infty}$ for $q>3$.  It does not, by itself,
imply
\begin{align*}
 \partial_t\vu+\vu\cdot\nabla\vu\in L^1(0,T;L^\infty(\Omega)),
 \qquad (R,Q,\alpha)\in L^\infty(0,T;W^{1,\infty}(\Omega)).
\end{align*}
Because these stronger bounds are required below, in the relative-entropy argument, one should assume one further
spatial derivative in the initial data and the corresponding boundary
compatibility conditions, and run the differentiated maximal-regularity
argument to obtain
\begin{align*}
 \vu&\in W^{1,p}(0,T;W^{1,q}(\Omega))\cap L^p(0,T;W^{3,q}(\Omega)),\\
 (R,Q,\alpha)&\in W^{1,p}(0,T;W^{2,q}(\Omega)).
\end{align*}
The embedding $W^{2,q}\hookrightarrow W^{1,\infty}$ then yields the bounds
used in the relative-energy estimates. This result can be anticipated, however, we need stronger assumptions on the domain ($\Omega \in C^3$) and initial data ($(R_0,Q_0,\alpha_0) \in (W^{2,q}(\Omega))^3$, $\vu_0 \in B^{3-2/p}_{q,p}(\Omega)$) including certain compatibility condition for the data. 
\end{rmk}

\subsection{Weak-strong uniqueness principle}

The second main result of this paper reads as follows

\begin{thm}\label{thm_entropy}
    Let $(\Nu_{t,x}, \mathcal{D})$ be a measure-valued solution to system \eqref{S}, with the boundary conditions \eqref{eq1.2bi} and the initial conditions \eqref{eq1.3bi} obtained i Theorem \eqref{t:mv}. Let $(\beta , \BR,\BQ,\vw)$ be the strong solution  emanating from the same initial data as in Remark \ref{r:stronger_reg}. Then the two solutions coincide on the time interval $(0,T)$ of a lifespan of the regular solution.
\end{thm}

To prove this theorem, we use the method of relative entropy, developed in the previous work of the authors for the system with algebraic closure. 
The suitable relative energy for the two-fluid system with a PDE closure \eqref{S} is
\eq{
\label{rel_entr}
E(\alpha,R,Q,\vu \,| \, \beta,\BR,\BQ,\vw)
=&
\frac12(R+Q) |\vu-\vw|^2 + (\alpha-\beta)^2  
\\
&
+
\alpha\Big(\Hp(\vrp)-\Hp'(\bvrp)(\vrp-\bvrp)-\Hp(\bvrp)
\Big) \\
&
+
(1-\alpha)\Big( \Hm(\vrm)-\Hm'(\bvrm)(\vrm-\bvrm)-\Hm(\bvrm)
\Big) ,
}
where for brevity we set 
\eq{\label{RQ_def}
R=\alpha\vrp,\quad Q=(1-\alpha)\vrm,\quad
\BR=\beta\bvrp,\quad \BQ=(1-\beta)\bvrm,
}
and
\begin{align}\label{H_def}
H_+(\vrp)=\frac{1}{\gp-1}\vrp^\gp,\quad \Hm(\vrm)=\frac{1}{\gm-1}\vrm^\gm.
\end{align}
The variables $\beta,\BR,\BQ,\vw$ stand for arbitrary, sufficiently smooth reference state satisfying
\eq{\label{reference}
0<\beta<1,\quad 0<r,Q<\infty,\quad \vw\Big|_{\partial\Omega}=\vc{0}.
}

The purpose of the second part of this work will be to show that the measure-valued generalization of this functional remains small provided it is small initially, and that the regular solution $(\beta, \BR, \BQ,\vw)$ emanating from the same initial data exists.


\section{Construction of dissipative measure-valued solution}
In this section we prove Theorem \ref{t:mv}. We start from solving the approximate system with parameters $\ep,n$ being fixed. In the further sections we recover the original system by letting the approximation parameters to the limit.

\subsection{Approximation scheme}
We approximate system \eqref{S} as follows. 
\begin{itemize}
\item
The continuity and the transport equations are regularized by means of standard parabolic regularization. For $\ep>0$, we have:
\begin{subequations}\label{Se}
\begin{align}
  &\pt R+ \Div(R\vu )=\ep\lap R,\label{SRe}\\ 
  &\pt Q + \Div(Q\vu )=\ep \lap Q ,\label{STe}\\
&\pt\alpha+\vu\cdot\Grad\alpha=\frac{(1-\alpha)\alpha }{\lambda+2\mu}\lr{\lr{\frac{R}{\alpha}}^\gp-\lr{\frac{Q}{1-\alpha}}^\gm}+\ep\lap\alpha.\label{Salphae}
\end{align}
\end{subequations} 
We supplement the system \eqref{Se} with the Neumann boundary conditions
\eqh{
\Grad Y\cdot\vc{n}\Big|_{\partial\Omega}=0,
}
as well as the regular initial conditions:
\eqh{
Y(0,x)=Y_{0,n}(x)\in C^{2,\sigma}(\overline\Omega),\ \sigma>0,\quad
\inf_\Omega Y_{0,n}>0,\quad \Grad Y_{0,n}\cdot\vc n\big|_{\partial\Omega}=0,
}
where for $Y$ we substitute $R,Q$ or $\alpha$. Furthermore,
$\sup_\Omega\alpha_{0,n}<1$.

\item The momentum equation \eqref{Su} is replaced by the Faedo--Galerkin approximation:
we look for the approximate solution $\vu\in  C^1([0,T];X_n)$, where $X_n$ is the linear hull of the first $n$ eigenvectors of the Laplace operator with Dirichlet boundary conditions, such that  the following
integral equation
\eq{ \label{Sue}
&\inttauO{\big((R_n+Q_n)\vu_n \cdot \partial_t \vcg{\varphi} + (R_n+Q_n)(\vu_n\otimes \vu_n): \Grad \vcg{\varphi} \big)}\\
= &\inttauO{ \vS(\Grad\vu_n):\nabla \vcg{\varphi}} -\inttauO{  p(\alpha_n,R_n,Q_n) \Div  \vcg{\varphi}}\\
&+\ep\inttauO{\Grad(R_n+Q_n)\cdot\Grad\vu_n\cdot \vcg{\varphi}}\\
&+ \intO{(R_n+Q_n)\vu_n \cdot \vcg{\varphi} } \Big|_{t=0}^{t=\tau}, 
}
holds for all $\tau\in(0,T]$ and for any test function $\vcg{\varphi}\in C^1([0,T];X_n)$, where
$R_{n,\ep}=R_{R_{0,n}}(\vu_{n,\ep})$ and
$Q_{n,\ep}=Q_{Q_{0,n}}(\vu_{n,\ep})$ are the unique solutions of
\eqref{SRe}, \eqref{STe}, while
$\alpha_{n,\ep}=\alpha_{\alpha_{0,n}}(\vu_{n,\ep},R_{n,\ep},Q_{n,\ep})$
is the unique solution of \eqref{Salphae}. We suppress the second index only
when $n$ and $\ep$ are fixed.
\item The initial parametrized probability measure
$\{\Nu_{0,x}\}_{x\in\Omega}$ satisfies  the assumptions of
Theorem~\ref{t:mv}. By the standard recovery argument for finite-energy
Young measures, there are smooth data
$(\alpha_{0,n},R_{0,n},Q_{0,n},\vu_{0,n})$, with
$0<\alpha_{0,n}<1$, $R_{0,n},Q_{0,n}>0$, and
$\vu_{0,n}\in X_n$, which generate the $(s,r,q,\vv)$-marginal of
$\Nu_{0,x}$ and satisfy
\[
 E(s,r,q,\vv):=\frac12(r+q)|\vv|^2+P(s,r,q),
\]
\begin{equation}\label{initial-recovery}
 \int_\Omega E(\alpha_{0,n},R_{0,n},Q_{0,n},\vu_{0,n})\,\dx
 \longrightarrow
 \int_\Omega\left\langle\Nu_{0,x};E(s,r,q,\vv)\right\rangle\,\dx .
\end{equation}
The approximating data are kept fixed and compatible with Neumann and Dirichlet boundary conditions  while $\ep\to0$. 
\end{itemize}

For $\ep>0$ and $n$ fixed, construction of solution $R_n,Q_n$  satisfying the above approximation and some uniform bounds follows the strategy from \cite{NM20}, which we roughly describe in the next subsection. We will therefore only emphasize the construction of solution to the volume fraction equation \eqref{Salphae}.



 \subsection{Solvability of the approximate system}
 
 In this section we take $\ep>0$, $n\in\mathbb N$ fixed. 

  \paragraph{Step I: Linearization.}
 Let us first fix $\vu=\vu_n\in C^1([0,T];X_n)$.  Since $X_n$ is
 spanned by smooth (Dirichlet) eigenfunctions of the Laplace operator with Dirichlet boundary conditions, $\vu$ is smooth in space,
 $\vu\cdot\vc n|_{\partial\Omega}=0$, and
 \[
 D:=\|\Div\vu\|_{L^\infty((0,T)\times\Omega)}<\infty .
 \]
 Given $\vu$ we will construct the functions $R,Q,\alpha$ solving \eqref{SRe}, \eqref{STe}, \eqref{Salphae}, respectively.
 \paragraph{Step II: The densities.}
 In advective form equation \eqref{SRe} reads 
 $$\pt R+\vu\cdot\Grad R-\ep\lap R=-R\,\Div\vu,$$
 and it is a linear uniformly parabolic equation with smooth bounded coefficients and
 Neumann data; classical theory (cf.\ \cite{AFN21}) yields existence of a unique
 $R\in C^{1,2}([0,T]\times\overline\Omega)$, and likewise $Q$ solving \eqref{SRe}, \eqref{STe}, respectively. The spatially
 constant functions $e^{Dt}\sup_\Omega R_{0,n}$ and $e^{-Dt}\inf_\Omega R_{0,n}$
 are super/sub-solutions, so the maximum principle gives
 \[
 0<e^{-Dt}\inf_\Omega R_{0,n}\le R\le e^{Dt}\sup_\Omega R_{0,n} ,
 \]
 and similarly for $Q$. Hence, with
 \[
 \underline m:=e^{-DT}\min\{\textstyle\inf_\Omega R_{0,n},\inf_\Omega Q_{0,n}\}>0,
 \qquad
 \overline M:=e^{DT}\max\{\textstyle\sup_\Omega R_{0,n},\sup_\Omega Q_{0,n}\},
 \]
 one has $\underline m\le R,Q\le\overline M$ on $[0,T]\times\Omega$.
 \paragraph{Step III: The volume fraction.}
 We write \eqref{Salphae} as
 \begin{equation}\label{alpha-eq}
 	\pt\alpha+\vu\cdot\Grad\alpha-\ep\lap\alpha=F(t,x,\alpha),\qquad
 	F(t,x,\alpha):=\frac{\alpha(1-\alpha)}{\lambda+2\mu}
 	\lr{\lr{\tfrac R\alpha}^{\gp}-\lr{\tfrac Q{1-\alpha}}^{\gm}}.
 \end{equation}
 As $R,Q$ are bounded and bounded below, $F(t,x,\cdot)\in C^1(0,1)$ but
 \begin{equation}\label{blowup}
 	\lim_{\alpha\to0^+}F=+\infty,\qquad \lim_{\alpha\to1^-}F=-\infty ,
 \end{equation}
 so $F$ fails to be Lipschitz up to $\{0,1\}$ and a direct application of
 parabolic existence theory  is not available. However, the singularity has the right sign, which we will make more precise below and we will use it to construct solution to \eqref{alpha-eq} with $F(\alpha)$ replaced by certain modification.
 
This is the reason why we assume that the initial data for the volume fraction satisfy
 \eqh{0<\inf_\Omega\alpha_{0,n}\le\sup_\Omega\alpha_{0,n}<1.}
 We will prove the following result.
 \begin{lemma}\label{lem:alpha}
 	There exist $\eta\in(0,\tfrac12)$, depending only on
 	$\underline m,\overline M,\gp,\gm,\alpha_{0,n}$, and a unique
 	$\alpha\in C^{1,2}([0,T]\times\overline\Omega)$ solving \eqref{alpha-eq} with
 	$\Grad\alpha\cdot\vc n|_{\partial\Omega}=0$, $\alpha(0)=\alpha_{0,n}$, and
 	\begin{equation}\label{eta}
 		0<\eta\le\alpha\le1-\eta<1\qquad\text{on }\ [0,T]\times\Omega .
 	\end{equation}
 \end{lemma}
 \begin{pf}
  {\emph{Step 1. Truncation. }}For $\delta\in(0,\tfrac12)$ we consider the following truncation of $F$:
 	\begin{equation}\label{trunc}
 		F_\delta(t,x,\alpha):=
 		\begin{cases}
 			F(t,x,\delta),& \alpha\le\delta,\\
 			F(t,x,s),& \delta\le \alpha\le1-\delta,\\
 			F(t,x,1-\delta),& \alpha\ge1-\delta .
 		\end{cases}
 	\end{equation}

      {\emph{Step 2. Existence for fixed $\delta$. }}
 	For $\alpha\in[\delta,1-\delta]$ the fractions $R/\alpha,\,Q/(1-\alpha)$ are bounded by
 	$\overline M/\delta$, so $F_\delta$ is bounded, $|F_\delta|\le C_\delta$, and
 	globally Lipschitz in $\alpha$, $|\partial_\alpha F_\delta|\le L_\delta$, uniformly in
 	$(t,x)$ (with $C_\delta,L_\delta\uparrow\infty$ as $\delta\downarrow0$).\\

For $\delta$ fixed, the nonlinearity  in the problem
 	\begin{equation}\label{trunc-eq}
 		\pt\alpha^\delta+\vu\cdot\Grad\alpha^\delta-\ep\lap\alpha^\delta
 		=F_\delta(t,x,\alpha^\delta),\quad
 		\Grad\alpha^\delta\cdot\vc n|_{\partial\Omega}=0,\quad
 		\alpha^\delta(0)=\alpha_{0,n},
 	\end{equation}
 	is globally Lipschitz and  H\"older continuous. Hence, by the standard
 	contraction/Schauder argument,  it admits  a unique
 	solution $\alpha^\delta\in C^{1,2}([0,T]\times\overline\Omega)$. The compatibility
 	$\Grad\alpha_{0,n}\cdot\vc n|_{\partial\Omega}=0$ together with
 	$\alpha_{0,n}\in C^{2,\sigma}(\overline{\Omega})$ makes the solution classical up to $t=0$.\\
    
{\emph{Step 3. Choice of $\eta$. }}
We now choose $\eta\in(0,\tfrac12)$ such that
 	\begin{equation}\label{eta-choice}
 		\eta\le\min\Big\{\inf_\Omega\alpha_{0,n},\;1-\sup_\Omega\alpha_{0,n},\;
 		\underline m\,(2\overline M)^{-\gm/\gp},\;
 		\underline m\,(2\overline M)^{-\gp/\gm}\Big\}.
 	\end{equation}
 	Using $1-\eta\ge\tfrac12$, the last two entries give
 	\begin{equation}\label{barrier}
 		\lr{\tfrac{\underline m}\eta}^{\gp}\ge\lr{\tfrac{\overline M}{1-\eta}}^{\gm},
 		\qquad
 		\lr{\tfrac{\underline m}\eta}^{\gm}\ge\lr{\tfrac{\overline M}{1-\eta}}^{\gp},
 	\end{equation}
 	which is achievable since the left sides tend to $\infty$ while the right sides
 	stay bounded as $\eta\to0$. With $R\ge\underline m,\ Q\le\overline M$
 	(resp.\ $R\le\overline M,\ Q\ge\underline m$) and the monotonicity of
 	$s\mapsto s^{\gp},s^{\gm}$, \eqref{barrier} yields
 	\begin{equation}\label{signcond}
 		F(t,x,\eta)\ge0,\qquad F(t,x,1-\eta)\le0\qquad\text{on }[0,T]\times\Omega .
 	\end{equation}
 	From now on fix $\delta\le\eta$; then $[\eta,1-\eta]\subseteq[\delta,1-\delta]$, so
 	$F_\delta(\cdot,\eta)=F(\cdot,\eta)\ge0$ and
 	$F_\delta(\cdot,1-\eta)=F(\cdot,1-\eta)\le0$.\\

 	{\emph{Step 4. Upper bound for $\alpha^\delta$. }} Let $W:=\alpha^\delta-(1-\eta)$. As $1-\eta$ is
 	constant, \eqref{trunc-eq} and \eqref{signcond} give
 	\[
 	\pt W+\vu\cdot\Grad W-\ep\lap W
 	=\underbrace{F_\delta(t,x,\alpha^\delta)-F_\delta(t,x,1-\eta)}_{=\,b\,W}
 	+\underbrace{F_\delta(t,x,1-\eta)}_{\le0}\le b\,W,
 	\]
 	where $b(t,x):=\int_0^1\partial_s F_\delta\big(t,x,(1-\eta)+\theta W\big)\,{\rm d}\theta$,
 	$|b|\le L_\delta$. Thus
 	\[
 	\pt W+\vu\cdot\Grad W-\ep\lap W-b\,W\le0,\quad
 	\Grad W\cdot\vc n|_{\partial\Omega}=0,\quad
 	W(0)=\alpha_{0,n}-(1-\eta)\le0 .
 	\]
 	The coefficient $b$ is not signed, to fix it we set $\widehat L_\delta:=L_\delta+D$ and
 	$V:=e^{-\widehat L_\delta t}W$. Dividing by $e^{\widehat L_\delta t}$,
 	\[
 	\pt V+\vu\cdot\Grad V-\ep\lap V+(\widehat L_\delta-b)\,V\le0,\qquad
 	\widehat L_\delta-b\ge D\ge0,\qquad V(0)\le0 .
 	\]
 	By the parabolic maximum principle (Neumann data, non-negative zeroth-order
 	coefficient) $V\le0$, i.e.\ $\alpha^\delta\le1-\eta$ on $[0,T]\times\Omega$.\\

{\emph{Step 5. Lower bound for $\alpha^\delta$.}} Similarly, let $\omega:=\eta-\alpha^\delta$. As $\eta$ is
 	constant, \eqref{trunc-eq} and \eqref{signcond} give
 	\[
 	\pt\omega+\vu\cdot\Grad\omega-\ep\lap\omega
 	=\underbrace{F_\delta(t,x,\eta)-F_\delta(t,x,\alpha^\delta)}_{=\,\tilde b\,\omega}
 	-\underbrace{F_\delta(t,x,\eta)}_{\ge0}\le\tilde b\,\omega,
 	\]
 	with $\displaystyle \tilde b(t,x):=\int_0^1\partial_s
 	F_\delta\big(t,x,\alpha^\delta+\theta\omega\big)\,{\rm d}\theta$, $|\tilde b|\le L_\delta$.
 	Hence
 	\[
 	\pt\omega+\vu\cdot\Grad\omega-\ep\lap\omega-\tilde b\,\omega\le0,\qquad
 	\omega(0)=\eta-\alpha_{0,n}\le0,
 	\]
 	and the same shift $e^{-\widehat L_\delta t}$ and maximum principle yield
 	$\omega\le0$, i.e.\ $\alpha^\delta\ge\eta$ on $[0,T]\times\Omega$. Together with
 	the upper bound this proves \eqref{eta} for $\alpha^\delta$, with $\eta$
 	independent of $\delta$.\\

{\emph{Step 6. Removing the truncation and uniqueness.}} Since $\delta\le\eta\le\tfrac12$, we have 
 	$[\eta,1-\eta]\subseteq[\delta,1-\delta]$. By the estimate by
 	\eqref{eta}, we therefore have $F_\delta\equiv F$ i.e.  the solution never meets the truncation, so $\alpha:=\alpha^\delta$
 	solves \eqref{alpha-eq}. 
    
    Regarding the uniqueness, the invariant strip $[\eta,1-\eta]\Subset(0,1)$, thus
 	$F(t,x,\cdot)$ is Lipschitz, so the difference of two solutions with equal data
 	satisfies a linear parabolic inequality with bounded zeroth-order coefficient
 	and zero initial datum; the maximum principle (or Gr\"onwall) gives uniqueness.  $\Box$
 \end{pf}

 \paragraph{Step IV: The fixed point argument.}
 Setting $R_n:=R$, $Q_n:=Q$, $\alpha_n:=\alpha$ provides the coefficients of
 \eqref{Sue} with the uniform bounds $\underline m\le R_n,Q_n\le\overline M$ and
 $\eta\le\alpha_n\le1-\eta$; in particular $\alpha_n(t,x)\in(0,1)$. The fixed point argument for local-in-time solutions of the Faedo--Galerkin scheme
 then follows as in the standard theory for the single-component Navier--Stokes equations, see \cite{NS}, Chapter 7, and \cite{NM20,Novotny19} for similar constructions. This solution then can be extended to a global in-time solution provided the energy estimate holds uniformly in time.

To prove this, we first use $\vu_n$ as a test function in \eqref{Sue}  and integrate by parts to obtain
\eqh{
&\inttauOB{\vr_n\frac12\pt|\vu_n|^2+\vr_n\vu_n\otimes \vu_n:\Grad\vun}\\
&-\inttauO{\Grad(\alpha_n\vrp_n^\gp+(1-\alpha_n)\vrm_n^\gm)\cdot\vun}\\
&\quad=\inttauOB{\vS(\vu_n):\Grad\vu_n+\frac12\ep\Grad\vr_n\cdot\Grad|\vun|^2}+ \intO{\vr_n|\vu_n|^2 } \Big|_{t=0}^{t=\tau},
}
where for brevity of notation we denoted
\eq{\label{notation}
\vr_n=R_n+Q_n,\quad \vrp_n=\frac{R_n}{\alpha_n},\quad \vrm_n\frac{Q_n}{(1-\alpha_n)}.}
Summing up equations \eqref{SRe} and \eqref{STe} and multiplying them by $\frac12|\vun|^2$ and integrating by parts, we obtain
\eqh{
&\inttauOB{-\frac12\vr_n\pt|\vun|^2-\frac12\vr_n\vun\cdot\Grad|\vun|^2+\frac12\ep\Grad\vr_n\cdot\Grad|\vun|^2}=-\frac12\intO{\vr_n|\vu_n|^2 } \Big|_{t=0}^{t=\tau}.
}
Summing up the two formulas, we obtain
\eqh{
&\frac12\intO{\vrn|\vun|^2}\Big|_0^\tau+\inttauO{\vS(\vun):\Grad\vun}=-\inttauO{\Grad(\alpha_n\vrp_n^\gp+(1-\alpha_n)\vrm_n^\gm)\cdot\vun}.
}
In order to proceed, let us analyze the last term
\eq{&\inttauO{\Grad(\alpha_n\vrp_n^\gp+(1-\alpha_n)\vrm_n^\gm)\cdot\vu_n}\\
&=
\inttauO{(\Grad\alpha_n\vrp_n^\gp+\Grad(1-\alpha_n)\vrm_n^\gm)\cdot\vu_n}\\
&+\inttauO{\frac{\gp}{\gp-1}\alpha_n\vrp_n\vu_n\cdot\Grad\vrp_n^{\gp-1}}
+\inttauO{\frac{\gm}{\gm-1}(1-\alpha_n)\vrm_n\vu_n\cdot\Grad\vrm_n^{\gm-1}}\\
&=
\inttauO{(-\pt\alpha_n\vrp_n^\gp-\pt(1-\alpha_n)\vrm_n^\gm)}+\inttauO{\frac{(1-\alpha_n)\alpha_n }{\lambda+2\mu}(\vrp_n^\gp-\vrm_n^\gm)^2}\\
&+\ep\inttauOB{\lap\alpha_n\vrp_n^\gp+\lap(1-\alpha_n)\vrm_n^\gm}\\
&+\inttauO{\frac{\gp}{\gp-1}\pt(\alpha_n\vrp_n)\vrp_n^{\gp-1}}
+\inttauO{\frac{\gm}{\gm-1}\pt((1-\alpha_n)\vrm_n)\vrm_n^{\gm-1}}\\
&-\ep\inttauO{\frac{\gp}{\gp-1}\lap(\alpha_n\vrp_n)\vrp_n^{\gp-1}}
-\ep\inttauO{\frac{\gm}{\gm-1}\lap((1-\alpha_n)\vrm_n)\vrm_n^{\gm-1}},
}
where we have subsequently used equations \eqref{Salphae},\eqref{SRe} and \eqref{STe}.
Now notice that we have
\eqh{
&\inttauO{(-\pt\alpha_n\vrp_n^\gp-\pt(1-\alpha_n)\vrm_n^\gm)}\\
&+\inttauO{\frac{\gp}{\gp-1}\pt(\alpha_n\vrp_n)\vrp_n^{\gp-1}}
+\inttauO{\frac{\gm}{\gm-1}\pt((1-\alpha_n)\vrm_n)\vrm_n^{\gm-1}}\\
&=\intOB{\frac{1}{\gp-1}\alpha_n\vrp_n^{\gp}+\frac{1}{\gm-1}(1-\alpha_n)\vrm_n^{\gm}}\Big|_0^\tau,
}
which is the missing part of the energy. As for the $\ep$-dependent terms, we see that we have
\eq{&\ep\inttauO{\lap\alpha_n\vrp_n^\gp}-\ep\inttauO{\frac{\gp}{\gp-1}\lap(\alpha_n\vrp_n)\vrp_n^{\gp-1}}\\
&=-\ep\inttauO{\Grad\alpha_n\cdot\Grad\vrp_n^\gp}+\ep\inttauO{\frac{\gp}{\gp-1}\Grad(\alpha_n\vrp_n)\cdot\Grad\vrp_n^{\gp-1}}\\
&=\ep\inttauO{\frac{\gp}{\gp-1}\alpha_n\Grad\vrp_n\cdot\Grad\vrp_n^{\gp-1}},}
and similarly for the phase ``-''. All in all, our energy estimate reads
\eqh{
&\frac12\intO{\vrn\vu_n^2}\Big|_0^\tau+\intOB{\frac{1}{\gp-1}\alpha_n\vrp_n^{\gp}+\frac{1}{\gm-1}(1-\alpha_n)\vrm_n^{\gm}}\Big|_0^\tau\\
&+\inttauO{\vS(\vu_n):\Grad\vu_n}+\inttauO{\frac{(1-\alpha_n)\alpha_n }{\lambda+2\mu}(\vrp_n^\gp-\vrm_n^\gm)^2}\\
&+\ep\inttauO{\gp\alpha_n\vrp_n^{\gp-2}|\Grad\vrp_n|^2}+\ep\inttauO{\gm(1-\alpha_n)\vrm_n^{\gm-2}|\Grad\vrm_n|^2}=0.
}
Coming back to our original notation \eqref{notation}, and retaining the dependence on both parameters $(\alpha_{0,n},R_{0,n},Q_{0,n},\vu_{0,n})$ this implies 
\eq{\label{en_en}
 &\int_\Omega E(\alpha_{n,\ep},R_{n,\ep},Q_{n,\ep},\vu_{n,\ep})(\tau)\,\dx
 +\int_0^\tau\!\!\int_\Omega\vS(\Grad\vu_{n,\ep}):\Grad\vu_{n,\ep}\,\dxdt\\
 &\quad+\int_0^\tau\!\!\int_\Omega
 \frac{\alpha_{n,\ep}(1-\alpha_{n,\ep})}{\lambda+2\mu}
 \left[\left(\frac{R_{n,\ep}}{\alpha_{n,\ep}}\right)^\gp
 -\left(\frac{Q_{n,\ep}}{1-\alpha_{n,\ep}}\right)^\gm\right]^2\dxdt
 \le\int_\Omega E(\alpha_{0,n},R_{0,n},Q_{0,n},\vu_{0,n})\,\dx,
}
where  $\tau$ is arbitrary in $(0,T]$.

We can summarize the findings of this section as follows.
\begin{prop}[Solvability of the approximate problem]\label{prop:approx-solvability}
Let $n\in\mathbb N$ and $\ep>0$. For the smooth compatible data above, the
regularized Galerkin problem (\ref{SRe}--\ref{initial-recovery}) admits a global solution
$(\alpha_{n,\ep},R_{n,\ep},Q_{n,\ep},\vu_{n,\ep})$ such that $$\alpha_{n,\ep},R_{n,\ep},Q_{n,\ep}\in C^{1,2}([0,T]\times\overline\Omega),$$
$$\vu_{n,\ep}\in C^1([0,T];X_n),$$ 
and
\[
 R_{n,\ep},Q_{n,\ep}>0,\qquad
 0<\eta_{n}\le\alpha_{n,\ep}\le1-\eta_{n}<1.
\]
The masses of $R_{n,\ep}$ and $Q_{n,\ep}$ are conserved and the energy
identity \eqref{en_en} holds for every $\tau\in(0,T]$; in particular, the solution
extends globally in time. 
\end{prop}

\begin{rmk}\label{rem:uniform-separation-fixed-n}
For fixed $n$, the energy bound and equivalence of norms on $X_n$ imply
\[
 \int_0^T\|\Div\vu_{n,\ep}\|_{L^\infty(\Omega)}\,\dt\le C(n)
\]
uniformly in $\ep$. Thus the maximum-principle bounds for
$R_{n,\ep},Q_{n,\ep}$ and the  constant $\eta$ in
Lemma~\ref{lem:alpha} can be chosen independently of $\ep$. These constants
need not be uniform in $n$; Section~4.4 uses only $0\le\alpha_n\le1$.
\end{rmk}

\subsection{Passage to the limit $\ep \to 0$}

In this subsection $n$ is fixed and the full notation
$(\alpha_{n,\ep},R_{n,\ep},Q_{n,\ep},\vu_{n,\ep})$ is retained. In
particular, the velocity also depends on $\ep$. The initial data
$(\alpha_{0,n},R_{0,n},Q_{0,n},\vu_{0,n})$ do not depend on $\ep$.

By Proposition~\ref{prop:approx-solvability}, Remark~\ref{rem:uniform-separation-fixed-n},
and equivalence of norms on $X_n$, there are constants $c(n)$, $C(n)$, and $\eta(n)$ independent
of $\ep\in(0,1)$, such that
\begin{equation}\label{fixed-n-bounds}
\begin{aligned}
 &0<c(n)\le R_{n,\ep},Q_{n,\ep}\le C(n),\qquad
 \eta(n)\le\alpha_{n,\ep}\le1-\eta(n),\\
 &\|\vu_{n,\ep}\|_{L^\infty(0,T;X_n)}
 +\|R_{n,\ep}\|_{L^\infty(0,T;W^{1,2}(\Omega))}
 +\|Q_{n,\ep}\|_{L^\infty(0,T;W^{1,2}(\Omega))}\\
 &\hspace{35mm}
 +\|\alpha_{n,\ep}\|_{L^\infty(0,T;W^{1,2}(\Omega))}\le C(n)
\end{aligned}
\end{equation}
and
\begin{equation}\label{fixed-n-second-order}
 \sqrt\ep\big(\|\Delta R_{n,\ep}\|_{L^2(Q_T)}
 +\|\Delta Q_{n,\ep}\|_{L^2(Q_T)}
 +\|\Delta\alpha_{n,\ep}\|_{L^2(Q_T)}\big)\le C(n).
\end{equation}
Indeed, testing \eqref{SRe} by $R_{n,\ep}$ and integrating by parts, we first obtain that
\begin{equation*}
	\frac{1}{2}\Dt \|R_{n,\ep}\|_{L^2}^2 + \ep \|\Grad R_{n,\ep}\|_{L^2}^2 = -\frac{1}{2}\intO{R_{n,\ep}^2 \Div \vu_{n,\ep}} \le \frac{1}{2}\|\Div \vu_{n,\ep}\|_{L^\infty} \|R_{n,\ep}\|_{L^2}^2.
\end{equation*}
So, by Gr\"onwall's inequality:
\begin{equation}\label{R_L2_bound}
	\|R_{n,\ep}\|_{L^\infty(0,T; L^2(\Omega))}^2 + 2\ep \|\Grad R_{n,\ep}\|_{L^2(0,T; L^2(\Omega))}^2 \le e^{\|\Div \vu_{n,\ep}\|_{L^\infty} T} \|R_{0,n}\|_{L^2(\Omega)}^2 =: C_0(n).
\end{equation}

\smallskip
\noindent Testing \eqref{SRe} with $-\lap R_{n,\ep}$ and integrating by parts, we similarly may show that 
\begin{equation*}
	\Dt \|\Grad R_{n,\ep}\|_{L^2}^2 +2\ep \|\lap R_{n,\ep}\|_{L^2}^2\le C_1(\vu_{n,\ep})\|\Grad R_{n,\ep}\|_{L^2}^2 + C_2(\vu_{n,\ep})\|R_{n,\ep}\|_{L^2}^2,
\end{equation*}
where $C_1(\vu_{n,\ep}) := 2\|\Grad \vu_{n,\ep}\|_{L^\infty(\Omega)} + \|\Div \vu_{n,\ep}\|_{L^\infty(\Omega)} + 1$ and $C_2(\vu_{n,\ep}) := \|\Grad \Div \vu_{n,\ep}\|_{L^\infty(\Omega)}^2$. By \eqref{R_L2_bound} and Gr\"onwall inequality:
\eqh{
	&\|\Grad R_{n,\ep}\|_{L^\infty(0,T; L^2(\Omega))} + 2\ep \|\lap R_{n,\ep}\|_{L^2(0,T; L^2(\Omega))}^2\\
    &\le \Big(\|\Grad R_{0,n}\|_{L^2}^2 + C_2(\vu_{n,\ep}) C_0(n) T\Big)^{1/2} e^{\frac{1}{2}C_1(\vu_{n,\ep})T} =: C(n).
}
Exactly the same can be repeated for $Q_{n,\ep}$ and for $\alpha_{n,\ep}$.  Indeed, note that the source term  in the $\alpha$ equation \eqref{Salphae} as well as it's derivatives, remain  bounded due to \eqref{fixed-n-bounds}. This is the only point where the separation of $\alpha_{n,\ep}$ from the ends $0,1$ is used.

The scalar equations also give
\begin{equation}\label{fixed-n-time}
 \|\partial_tR_{n,\ep}\|_{L^2(Q_T)}+
 \|\partial_tQ_{n,\ep}\|_{L^2(Q_T)}+
 \|\partial_t\alpha_{n,\ep}\|_{L^2(Q_T)}\le C(n).
\end{equation}

Finally, it follows from the lower bounds in \eqref{fixed-n-bounds} and the projected momentum equation \ref{Sue} that $\{\vu_{n,\ep}\}_{\ep>0}$ is equicontinuous in
$C([0,T];X_n)$.

The Aubin--Lions compactness theorem, \eqref{fixed-n-time}, and the
finite-dimensional compactness yield, along a subsequence,
\begin{align}\label{fixed-n-convergence}
 &\vu_{n,\ep}\longrightarrow\vu_n
 &&\text{in }C([0,T];X_n),\nonumber\\
 &(R_{n,\ep},Q_{n,\ep},\alpha_{n,\ep})\longrightarrow(R_n,Q_n,\alpha_n)
 &&\text{in }C([0,T];L^2(\Omega))^3
\end{align}
and almost everywhere in $(0,T)\times\Omega$. Interpolation with the pointwise bounds in
\eqref{fixed-n-bounds} improves the scalar convergence to
$C([0,T];L^p(\Omega))$ for every finite $p$. Hence all pressure and
relaxation terms converge strongly in every finite $L^p((0,T)\times\Omega)$.

Moreover, \eqref{fixed-n-second-order} gives
\[
 \ep\Delta R_{n,\ep},\quad \ep\Delta Q_{n,\ep},\quad
 \ep\Delta\alpha_{n,\ep}\longrightarrow0\quad\text{in }L^2((0,T)\times\Omega).
\]

The artificial correction in the momentum equation also vanishes: after
integration by parts against a fixed $\vcg\varphi\in X_n$ it is bounded by
$C(n)\ep\|R_{n,\ep}+Q_{n,\ep}\|_{L^1(Q_T)}$. We therefore obtain
\begin{align}\label{galerkin-limit}
 &\partial_tR_n+\Div(R_n\vu_n)=0,\qquad
 \partial_tQ_n+\Div(Q_n\vu_n)=0,\nonumber\\
 &\partial_t\alpha_n+\vu_n\cdot\Grad\alpha_n
 =\frac{\alpha_n(1-\alpha_n)}{\lambda+2\mu}
 \left[\left(\frac{R_n}{\alpha_n}\right)^\gp
 -\left(\frac{Q_n}{1-\alpha_n}\right)^\gm\right],
\end{align}
together with the Galerkin momentum equation 
\eq{ \label{Sue:2}
&\inttauO{\big((R_n+Q_n)\vu_n \cdot \partial_t \vcg{\varphi} + (R_n+Q_n)(\vu_n\otimes \vu_n): \Grad \vcg{\varphi} \big)}\\
= &\inttauO{ \vS(\Grad\vu_n):\nabla \vcg{\varphi}} -\inttauO{  p(\alpha_n,R_n,Q_n) \Div  \vcg{\varphi}}\\
&+ \intO{(R_n+Q_n)\vu_n \cdot \vcg{\varphi} } \Big|_{t=0}^{t=\tau}. 
}

For $b\in C^1([0,1])$, the parabolic chain rule gives
\[
 \partial_tb(\alpha_{n,\ep})+\vu_{n,\ep}\cdot\Grad b(\alpha_{n,\ep})
 =b'(\alpha_{n,\ep})F_{n,\ep}
 +\ep b'(\alpha_{n,\ep})\Delta\alpha_{n,\ep}.
\]
The last term tends to zero directly in $L^2(Q_T)$, since
\[
 \|\ep b'(\alpha_{n,\ep})\Delta\alpha_{n,\ep}\|_{L^2(Q_T)}
 \le\|b'\|_{L^\infty}\sqrt\ep
 \big\|\sqrt\ep\Delta\alpha_{n,\ep}\big\|_{L^2(Q_T)}
 \longrightarrow0.
\]
Thus \eqref{galerkin-limit} holds in the renormalized sense.
Finally, lower semicontinuity and the strong convergence of all scalar
terms give, for every $\tau\in[0,T]$,
\begin{align}\label{energy-n}
 &\int_\Omega E(\alpha_n,R_n,Q_n,\vu_n)(\tau)\,\dx
 +\int_0^\tau\!\!\int_\Omega\vS(\Grad\vu_n):\Grad\vu_n\,\dxdt\nonumber\\
 &\quad+\int_0^\tau\!\!\int_\Omega
 \frac{\alpha_n(1-\alpha_n)}{\lambda+2\mu}
 \left[\left(\frac{R_n}{\alpha_n}\right)^\gp
 -\left(\frac{Q_n}{1-\alpha_n}\right)^\gm\right]^2\dxdt
 \le\int_\Omega E(\alpha_{0,n},R_{0,n},Q_{0,n},\vu_{0,n})\,\dx.
\end{align}

\subsection{Passage to the limit $n\to\infty$}
We now work with the sequence $(R_n,Q_n,\alpha_n,\vu_n)$ constructed in the previous subsection.

Since the energy estimate \eqref{energy-n} is uniform in
$n$, we have
\begin{align}\label{uniform-n-bounds}
 &
 \|R_n\|_{L^\infty(0,T;L^\gp(\Omega))}+
 \|Q_n\|_{L^\infty(0,T;L^\gm(\Omega))}\le C,\nonumber\\
 &\|\sqrt{R_n+Q_n}\,\vu_n\|_{L^\infty(0,T;L^2(\Omega))}
 +\|\vu_n\|_{L^2(0,T;W^{1,2}_0(\Omega))}\le C,\nonumber\\
 &\left\|\frac{\alpha_n(1-\alpha_n)}{\lambda+2\mu}
 \left[\left(\frac{R_n}{\alpha_n}\right)^\gp
 -\left(\frac{Q_n}{1-\alpha_n}\right)^\gm\right]^2
 \right\|_{L^1(Q_T)}\le C.
\end{align}
In particular,
\begin{equation}\label{mass-flux-bound}
 \|R_n\vu_n\|_{L^\infty(0,T;L^{2\gp/(\gp+1)}(\Omega))}
 +\|Q_n\vu_n\|_{L^\infty(0,T;L^{2\gm/(\gm+1)}(\Omega))}\le C.
\end{equation}
Both exponents are greater than one for $\gp,\gm>1$; thus the two mass fluxes are equi-integrable. 

Unlike the energy estimate the bounds for $\alpha_{n}$ in \eqref{fixed-n-bounds} do no longer guarantee that $\alpha_n$ is far away from the values $0$ and $1$. However, bu  \eqref{eta-choice}, $0<\eta(n)<\tfrac12<1$, and so Lemma~\ref{lem:alpha} gives
$0<\eta(n)\le\alpha_n\le1-\eta(n)<1$ on $[0,T]\times\Omega$; in particular 
$$0\le\alpha_n\le1$$ 
independently of $n$.
It is thus clear that $\alpha_n$ might touch the values $0$ and $1$ in the limit and it is therefore important to extend the definitions of all pressure-related terms that become singular at these values.
Let us denote
\begin{align*}
 \omega(s,r,q)&:=\frac{s(1-s)}{\lambda+2\mu}
 \big(s^{-\gp}r^\gp-(1-s)^{-\gm}q^\gm)\big),\\
 \mathcal R(s,r,q)&:=\frac{s(1-s)}{\lambda+2\mu}
 \big(s^{-\gp}r^\gp-(1-s)^{-\gm}q^\gm)\big)^2,
\end{align*}
and recall the definitions
\begin{align*}
 P(s,r,q)&=\frac{s^{1-\gp}r^\gp}{\gp-1}
 +\frac{(1-s)^{1-\gm}q^\gm}{\gm-1},\\
 p(s,r,q)&=s^{1-\gp}r^\gp+(1-s)^{1-\gm}q^\gm.
\end{align*}
Out of the four functions above, three last ones are nonnegative and can be extended to lower semicontinuous functions on $\mathbb{R}^3$ by assigning value $0$ at the ``compatible'' points $(0,0,q)$ and $(1,r,0)$ (where the phase occupying zero volume has a zero partial mass)  and value $+\infty$ at the incompatible points $(1,0,q)$ and $(0,r,0)$. For example, for $ \mathcal R$ we have
\begin{align}\label{R-lsc-ext}
    (s,r,q)\in\R^{3}\ \longmapsto\ \mathcal R(s,r,q)
=\begin{cases}
\dfrac{s(1-s)}{\lambda+2\mu}
\left[\left(\dfrac{r}{s}\right)^{\gp}-\left(\dfrac{q}{1-s}\right)^{\gm}\right]^{2}
& \text{if } 0<s<1,\ r,q\ge0,\\[10pt]
0 & \text{if } s=0,\ r=0,\ q\ge0,\\[4pt]
0 & \text{if } s=1,\ q=0,\ r\ge0,\\[4pt]
+\infty & \text{otherwise.}
\end{cases} .
\end{align}
The pressure potential $P(s,r,q)$ as well as the pressure $p(s,r,q)$ can be extended in the same way.
On the other hand, the relaxation term $ \omega(s,r,q)$ is sign-changing and does not
admit an $[0,+\infty]$-valued lower semicontinuous \ extension; it is instead controlled pointwisely by $\mathcal R(s,r,q)$, we have
\eq{\label{oR}
|\omega|\le\Big(\tfrac{s(1-s)}{\lambda+2\mu}\Big)^{1/2}\sqrt{\mathcal R}
\le\tfrac{1}{2\sqrt{\lambda+2\mu}}\sqrt{\mathcal R},
}
whence, by \eqref{uniform-n-bounds} it is equintegrable 
\eq{\label{bound_om}
\|\omega(\alpha_n,R_n,Q_n)\|_{L^2(Q_T)}\leq C.
}


\subsubsection{Passage to the limit in the equations for $R_n, Q_n$ and  $b(\alpha_n)$}

The sequence
$(\alpha_n,R_n,Q_n,\vu_n,\mathbb D(\vu_n))$ generates a parametrized Young
measure $\{\Nu_{t,x}\}$ on $\mathcal F$. The bounds
\eqref{uniform-n-bounds}--\eqref{mass-flux-bound} give the continuity
equations \eqref{mv eqn 1r}--\eqref{mv eqn 1q} without concentration
defects, and weak convergence in the integration-by-parts identity
\begin{align*}
    	-\inttauO{\vu_n\cdot\Div\mathbb M }=	\inttauO{ \mathbb{D}(\Grad \vu_n):\mathbb M},
\end{align*}
 for any $\displaystyle \mathbb{M} \in C^1(\bar{Q}_T; \mathbb{R}^{3\times 3}_{\text{sym}}) $,  gives \eqref{mv-gradient-compatibility}.

It remains to justify the passage to the limit in the renormalized equation
\begin{align}\label{ren_n}
    \pt b(\alpha_n)+\Div\lr{b(\alpha_n) \vu_n}- b(\alpha_n)\Div\vu_n
	=b'(\alpha_n)\,\omega(\alpha_n,R_n,Q_n),
\end{align}
in particular the behavior of the source term near
$\alpha=0,1$. 

It follows from \eqref{oR} that at  endpoints $(0,0,q)$ and $(1,r,0)$ where $\omega(s,r,q)$  is not defined a-priori, we have $\omega=0$. Indeed, since $\mathcal R=0$ on $\{s=0,r=0\}\cup\{s=1,q=0\}$ so is $\omega$.
Since $\lel\Nu_{t,x};\mathcal R(s,r,q)\ril<\infty$ for a.e.\ $(t,x)$, one has
$\Nu_{t,x}\big(\{s=0,r>0\}\cup\{s=1,q>0\}\big)=0$, so $\lel\Nu_{t,x};\omega(s,r,q)\ril$ is well
defined a.e..\\
Now, due to \eqref{bound_om} 
$b'(\alpha_n)\omega(\alpha_n,R_n,Q_n)$  is uniformly bounded in $L^2((0,T)\times\Omega)$ for any $b\in C^1([0,\infty)])$. Thus, we have
\[
b'(\alpha_n)\,\omega(\alpha_n,R_n,Q_n)\ \rightharpoonup\
\lel\Nu_{t,x};b'(s)\omega(r,s,q)\big)\ril
\qquad\text{weakly in }L^2((0,T)\times\Omega).
\]

On the other hand, the uniform bound
\[
\|b(\alpha_n)\Div\vu_n\|_{L^2_{t,x}}\le\|b(\alpha_n)\|_{L^\infty_{t,x}}\|\Div\vu_n\|_{L^2_{t,x}}\le C,
\]
implies that 
\eqh{b(\alpha_n)\Div\vu_n\rightharpoonup\lel\Nu_{t,x};b(s)\,\tr\mathbb{D}_{\vv}\ril\quad \text{ weakly in }L^2((0,T)\times\Omega).}
Thus, passing to the limit $n\to \infty$ in the weak formulation of \eqref{ren_n}, to verify \eqref{mv eqn 3}.


\subsubsection{The energy defects}

For completeness we record the lower-semicontinuity fact used for the
energy concentrations.
\begin{lemma}[Positive lower-semicontinuous defect]\label{lem:lsc-defect}
Let $Z_n$ generate the Young measure $\Nu$ and let
$f\colon\mathcal F\to[0,+\infty]$ be lower semicontinuous. If
$f(Z_n)\stackrel{*}{\rightharpoonup}\overline f$ in the space of Radon
measures, then
\[
 \overline f-\langle\Nu;f\rangle\,\dxdt\quad\text{is a nonnegative measure}.
\]
The same assertion holds for almost every time slice of an
$L^\infty$-in-time sequence.
\end{lemma}
\noindent
{When $f\colon\mathcal F\to[0,+\infty]$ is continuous, the above Lemma is proved by Feireisl et. al \cite[Lemma 2.1]{FG2W}.} For $f$ lower semicontinuous we write $f$ as the increasing supremum of bounded continuous
functions $f_k$. The definition of generation gives the limit for $f_k$,
while $f_k\le f$ gives
$\langle\Nu;f_k\rangle\,\dxdt\le\overline f$; monotone convergence as
$k\to\infty$ proves the claim. {For the details we refer to \cite[Chapter 5, Cor. 5.1.13 ]{FLMSBook} and \cite[Lemma 4.2]{C2021}.}

\noindent By \eqref{energy-n} we can extract subsequences such that
\begin{align*}
	&\Big[\tfrac12\vr_n|\vu_n|^2+P(\alpha_n,R_n,Q_n)\Big](t,\cdot)
	\xrightarrow{\,*\,}\Ov{E}
	&&\text{in }L^\infty_{w*}\big(0,T;\mathcal M(\overline\Omega)\big),\\
	&\vS(\mathbb{D}(\vu_n):\mathbb{D}(\vu_n)\xrightarrow{\,*\,}\Ov{\vS:\DD}
	&&\text{in }\mathcal M^+\big([0,T]\times\overline\Omega\big),\\
	&\tfrac{\alpha_n(1-\alpha_n)}{\lambda+2\mu}\big(\vrp_n^\gp-\vrm_n^\gm\big)^2
	\xrightarrow{\,*\,} \Ov{\mathcal R}
	&&\text{in }\mathcal M^+\big([0,T]\times\overline\Omega\big),
\end{align*}
with concentration defects (limit minus barycenter)
\begin{align*}
 \mathcal E_\infty(t)&:=\overline E(t)
 -\langle\Nu_{t,x};E(s,r,q,\vv)\rangle\,\dx,\\
 \sigma_\infty&:=\overline{\vS:\mathbb D}
 -\langle\Nu_{t,x};\vS(\mathbb D_\vv):\mathbb D_\vv\rangle\,\dxdt,\\
 \mathcal R_\infty&:=\overline{\mathcal R}
 -\langle\Nu_{t,x};\mathcal R(s,r,q)\rangle\,\dxdt.
\end{align*}
We want to use Lemma \ref{lem:lsc-defect} to show that those defects are non-negative. Indeed, as discussed before, thanks to extension \eqref{R-lsc-ext} of ${\cal R}(s,q,r)$ ad analogous extension of $P(s,q,r)$, both functions are lower semicontinuous on $\mathbb{R}^3$. The lower semicontinuity of the kinetic part of energy $\tfrac12(r+q)|\vv|^2$ and of the viscous dissipation $\vS(\mathbb D_\vv):\mathbb D_\vv$, is clear.
 Therefore, by Lemma \ref{lem:lsc-defect} we verify that
 \[
\mathcal E_\infty\ge0,\qquad \sigma_\infty\ge0,\qquad \mathcal R_\infty\ge0.
\]

This allows us to set the total nonnegative defect in energy inequality as
\[
\mathcal D(\tau):=\mathcal E_\infty(\tau)(\overline\Omega)
+\sigma_\infty\big([0,\tau]\times\overline\Omega\big)
+\mathcal R_\infty\big([0,\tau]\times\overline\Omega\big)\in L^\infty(0,T),\qquad \mathcal D\ge0.
\]

\medskip 

We can now pass to the limit in the energy inequality.  Due to \eqref{initial-recovery}, there is no initial defect: $\mathcal E_\infty(0)=0$, i.e.,\
$$\Ov{E}(0)(\overline\Omega)=\int_\Omega\lel\Nu_{0,x};\tfrac12(r+q)|\vv|^2+P(r,p,q)\ril\,dx.$$

Then passing to the limit in \eqref{energy-n} we get  for a.e.\ $\tau\in(0,T)$,
\begin{align*}
	&\intO{\lel\Nu_{\tau,x};\tfrac12(r+q)|\vv|^2+P(s,r,q)\ril}
	+\inttauO{\lel\Nu_{t,x};\tfrac{s(1-s)}{\lambda+2\mu}\big(p_+(r/s)-p_-(q/(1-s))\big)^2\ril}\\
	&\qquad+\inttauO{\lel\Nu_{t,x};\vS(\mathbb{D}_{\vv}):\mathbb{D}_{\vv}\ril}+\mathcal D(\tau)
	\le\intO{\lel\Nu_{0,x};\tfrac12(r+q)|\vv|^2+P(s,r,q)\ril}.
\end{align*}

\subsubsection{Passage to the limit in the momentum equation}

The tensor $\vr_n\vu_n\otimes\vu_n$ and the pressure $p(\alpha_n,R_n,Q_n)$ are bounded
only in $L^\infty(I;L^1(\Omega))$, the non-reflexive endpoint, hence concentrate in the
limit. Up to a subsequence, we have
\eqh{
&\vr_n\vu_n\otimes\vu_n\xrightarrow{\,*\,} \Ov{\vr\vu\otimes\vu},\\
& p(\alpha_n,R_n,Q_n) \xrightarrow{\,*\,} \Ov{p(\alpha,R,Q)},
}
in $L^\infty_{w*}(0,T;\mathcal M(\overline\Omega))$.

The \emph{momentum concentration defect} is defined as 
\begin{align*}
r^M:=&\Big(\Ov{\vr\vu\otimes\vu}-\lel\Nutx;(r+q)\vv\otimes\vv\ril\Big)
+\Big(\Ov{p(\alpha,R,Q)}-\lel\Nutx;p(s,r,q)\ril\Big)\mathbb I
\ \\
&\in\ L^\infty_{w*}\big(0,T;\mathcal M(\overline\Omega;\R^{d\times d})\big).
\end{align*}
Using this definition, and observation that the rest of the integrands in \eqref{Sue:2} is bounded in $L^p((0,T)\times\Omega)$ with some $p>1$, we can pass to the limit in \eqref{Sue:2} and obtain \eqref{mv eqn 2}.

\smallskip
\noindent We now want to show that $r^M$ can be dominated by the energy defect $\mathcal D(\tau)$.
\begin{itemize}
    \item Since $|(r+q)\vv\otimes\vv|\le2E(s,r,q,\vv)$, the standard concentration comparison
lemma \cite[Lemma~2.1]{FG2W} controls the kinetic part of $r^M$ by
$C\mathcal E_\infty$.
\item  For the pressure term set
$C_p=\max\{\gp-1,\gm-1\}$. Both $p$ and $C_pP-p$ are nonnegative lower
semicontinuous functions. Lemma~\ref{lem:lsc-defect} therefore gives
\[
 0\le\overline p-\langle\Nu;p(s,r,q)\rangle
 \le C_p\big(\overline P-\langle\Nu;P(s,r,q)\rangle\big)\leq C_p E_\infty.
\]
Here $\Ov{p},\ \overline P$ denotes the weak--$*$ measure limit of $p(\alpha_n,R_n,Q_n)$,
$P(\alpha_n,R_n,Q_n)$, respectively.
\end{itemize}

Consequently, for almost every $t$,
\begin{equation}\label{rM-energy-control}
 |r^M(t)|(\overline\Omega)\le C\mathcal E_\infty(t)(\overline\Omega)
 \le C\mathcal D(t),
\end{equation}
which is the compatibility condition \eqref{mom-def}. 

It remains to verify \eqref{mv-poincare-korn}. Fix
$\vw\in L^2(0,T;W^{1,2}_0(\Omega;\R^3))$. The sequence
$(\vu_n,\mathbb D\vu_n)$ generating $\Nu$ is bounded in $L^2$, and for every
$1<\beta<2$ the conformal Korn--Poincar\'e inequality gives
\begin{align*}
 \inttauO{|\vu_n-\vw|^\beta+|\mathbb D\vu_n-\mathbb D\vw|^\beta}
 \leq C_\beta\inttauO{
 |\mathbb T(\mathbb D\vu_n)-\mathbb T(\nabla\vw)|^\beta}.
\end{align*}
Both integrands have subquadratic growth. Therefore the fundamental theorem
for Young measures applies without a concentration measure, and passage to
$n\to\infty$ yields the corresponding inequality with the two integrands
replaced by their $\Nu_{t,x}$-averages. Letting $\beta\nearrow2$ gives
\eqref{mv-poincare-korn}. This is the argument of
\cite[Lemma~2.2]{BFN}, combined with the full $L^\beta$ conformal Korn
estimate \cite{LewintanNeff21}; the preliminary choice $\beta<2$ prevents an
additional quadratic concentration defect from appearing on the right-hand
side.

This finishes the proof of Theorem \ref{t:mv}. $\Box$

\medskip

We conclude this section with a few remarks.

\begin{rmk}
Let $\pi_s\colon\mathcal F\to[0,1]$, $(s,r,q,\vv,\mathbb{D}_{\vv})\mapsto s$, be the
projection onto the volume-fraction variable, with pushforward
$(\pi_s)_{\#}\Nu_{t,x}\in\mathcal P([0,1])$ characterized by
$$\langle(\pi_s)_{\#}\Nu_{t,x};g\rangle=\langle\Nu_{t,x};g\circ\pi_s\rangle$$ for
$g\in C([0,1])$. Since $s\in[0,1]$ on $\mathcal F$, this measure is supported in
$[0,1]$, and its barycenter
$$\alpha=\lel\Nu_{t,x};s\ril=\int_{[0,1]}s\,d(\pi_s)_{\#}\Nu_{t,x}$$ inherits
$0\le\alpha\le1$ a.e., so $\|\alpha\|_{L^\infty(Q_T)}\le1$.
\end{rmk}

\begin{rmk}
Relation between $p$ and $P$ plays a similar role to relation between $\omega$ and ${\cal R}$. In particular, the uniform
$L^1$-bound on the potential $P$, inherited in the limit (recall that 
$\lel\Nu_{t,x};P(r,s,q)\ril\in L^\infty_{w*}(0,T;\mathcal M(\overline\Omega))$) enforces 
\[
\Nu_{t,x}\big(\{s=0,\,r>0\}\big)=\Nu_{t,x}\big(\{s=1,\,q>0\}\big)=0
\qquad\text{for a.e. }(t,x).
\]
Thus $s=0\Rightarrow r=0$ and $s=1\Rightarrow q=0$ $\Nu_{t,x}$-a.s., and the
barycentric pressure $\lel\Nu_{t,x};p(s,r,q)\ril$ is well defined.
\end{rmk}

\section{Derivation of the relative entropy inequality}\label{Sec:MV-WS}

In this section we explain how the relative entropy \eqref{rel_entr}  can be generalized for  the measure-valued solutions introduced in Definition \ref{MVdef}.
Recall that our dummy variables are $(s,r,q,\vv,\mathbb{D}_\vv)$ along with the notation \eqref{Psrq}, \eqref{npnm}, \eqref{HpHm}.

For a measure-valued solution $(\Nu,{\mathcal D})$ and any sufficiently smooth reference state $(\beta,\BR,\BQ,\vw)$ satisfying \eqref{reference}, we define

\eq{\label{mv_rel_entr_b}
&{\cal E}_{mv}(\Nu\,|\,\beta,\BR,\BQ,\vw)(\tau)\\
&:=\intO{\lel \Nu_{\tau,x};
\frac12(r+q)|\vv-\vw|^2\ril}+\intO{\lel\Nu_{\tau,x}; (s-\beta)^2\ril}\\
&\quad+\intO{\lel \Nu_{\tau,x};
 s\Big(\Hp(\np)-\Hp'(\bvrp)(\np-\bvrp)-\Hp(\bvrp)\Big)\ril}\\
&\quad+\intO{\lel \Nu_{\tau,x};
 (1-s)\Big(\Hm(\nm)-\Hm'(\bvrm)(\nm -\bvrm)-\Hm(\bvrm)\Big)\ril}.
}
Notice that this simply means that the new relative entropy is related to \eqref{rel_entr} through
\eqh{
{\cal E}_{mv}(\tau)=\intO{\lel \Nu_{\tau,x};E(s,r,q,\vv|\beta,\BR,\BQ,\vw)\ril}.
}

Our first goal will be to derive some general formulation for part of the relative entropy excluding the volume-fraction component $\intO{\lel\Nu_{\tau,x}; (s-\beta)^2\ril}$.
\subsection{Relative entropy-preliminaries}
In this section we consider
\eq{\label{mv_rel_entr0}
&\bar{\cal E}_{mv}(\Nu\,|\,\beta,\BR,\BQ,\vw)(\tau)\\
&:=\intO{\lel \Nu_{\tau,x};
\frac12(r+q)|\vv-\vw|^2\ril}\\
&\quad+\intO{\lel \Nu_{\tau,x};
 s\Big(\Hp(\np)-\Hp'(\bvrp)(\np-\bvrp)-\Hp(\bvrp)\Big)\ril}\\
&\quad+\intO{\lel \Nu_{\tau,x};
 (1-s)\Big(\Hm(\nm)-\Hm'(\bvrm)(\nm -\bvrm)-\Hm(\bvrm)\Big)\ril}.
}
Notice that this is not exactly \eqref{mv_rel_entr_b}  as the   term $\intO{\lel\Nu_{\tau,x}; (s-\beta)^2\ril}$ is still missing and will be added later on.

We first expand \eqref{mv_rel_entr0} as follows: 
\eq{\label{mv_rel_entr}
&\bar{\cal E}_{mv}(\Nu\,|\,\beta,\BR,\BQ,\vw)(\tau)\\
&:=\intO{\lel \Nu_{\tau,x};
\frac12(r+q)|\vv|^2+s\Hp(\np)+(1-s)\Hm(\nm)\ril}\\
& \quad 
+\intO{\frac12(R+Q)|\vw|^2}\\
&\quad-\intO{\lel \Nu_{\tau,x};(r+q)\vv\ril \cdot \vw}\\
& \quad 
-\intO{R \Hp'(\bvrp)
}+\intO{\alpha\lr{\Hp'(\bvrp)\bvrp-\Hp(\bvrp)}}\\
& \quad 
-\intO{Q \Hm'(\bvrm)
}+\intO{(1-\alpha)\lr{\Hm'(\bvrm)\bvrm-\Hm(\bvrm)}}.
}
We now estimate each of the lines of this expression.

\medskip 
\noindent
{\it Step 1. } The first line is controlled by the energy inequality in Definition \ref{MVdef}, i.e.
		\begin{align*}
			\begin{split}
				&\int_{ \Omega} \lel \Nu_{t ,x}; \frac{1}{2}  (r+q)|\vv|^2 + P(s,r,q) \ril \dx\Big|_0^\tau\\ 
&\leq-\frac{1}{\lambda+2\mu}\inttauO{\lel \Nu_{t,x}; s(1-s)(p_{+}(\np) -p_{-}(\nm))^2\ril }\\
       &\quad-\inttauO{\lel\Nu_{t,x};\vS(\mathbb{D}_{\vv}): \mathbb{D}_{\vv}\ril}- \mathcal{D}(\tau) 
			\end{split}
		\end{align*}
\medskip

\noindent
{\it Step 2. } We sum the equations for $Q$ and $R$ and test the obtained result by $\frac12|\vw|^2$
\eq{\label{mv rhov2}
\intO{\frac12(R+Q)|\vw|^2}\Big|_0^\tau =
\inttauO{(R+Q)\vw\cdot\pt\vw}+\inttauO{\lel \Nu_{\tau ,x};(r+q)\vv\ril\cdot\Grad\vw\cdot \vw}.
} 

\noindent
{\it Step 3. } We take $\vw$ as a test function in the momentum equation \eqref{mv eqn 2}, and integrate by parts to get
\eq{\label{mv uv}
&-\intO{\lel \Nu_{t ,x};(r+q)\vv\ril\cdot\vw}\Big|_0^\tau\\
&=-\inttauO{\lel \Nu_{t ,x};(r+q)\vv\ril\cdot\pt\vw} -\inttauO{\lel \Nu_{t ,x};(r+q)\vv\otimes\vv\ril:\Grad\vw} \\
&
\quad-
\inttauO{\lel \Nu_{t ,x};p(s,r,q)\ril\Div\vw}\\
& \quad+ \inttauO{\lel\Nu_{t,x};\vS(\mathbb{D}_{\vv}):  \Grad \vw\ril}- \lel r^M;\Grad \vw\ril_{\mathcal{M}([0,\tau]\times \overline{\Omega}), C([0,\tau]\times \overline{\Omega})}.
}


\noindent
{\it Step 4.} Finally, we test the equation for $R$ by $\Hp'(\bvrp)$ and the equation for $Q$ by $\Hm'(\bvrm)$ to get
\eq{   \label{mv step4+}
-\inttauO{R\Hp'(\bvrp)}\Big|_0^\tau=-\inttauO{ \left[ R  \partial_{t}\Hp'(\bvrp) + \lel \Nu_{t,x};  r\vv\ril \cdot \Grad \Hp'(\bvrp)\right] }
}
\eq{   \label{mv step4-}
-\inttauO{Q\Hm'(\bvrm)}\Big|_0^\tau=-\inttauO{ \left[ Q  \partial_{t}\Hm'(\bvrm) + \lel \Nu_{t,x};  q\vv\ril \cdot \Grad \Hm'(\bvrm)\right] }.
}
\noindent
{\it Step 5.} Using that $\Hp'(\bvrp)\bvrp-\Hp(\bvrp)=\Pp(\bvrp)$, we obtain
\begin{align*}
\intO{\alpha\lr{\Hp'(\bvrp)\bvrp-\Hp(\bvrp)}}\Big|_0^\tau=\intO{\alpha\Pp(\bvrp)}\Big|_0^\tau.
\end{align*}
In a similar manner
\begin{align*}
\intO{(1-\alpha)\lr{\Hm'(\bvrm)\bvrm-\Hm(\bvrm)}}\Big|_0^\tau
=
\intO{(1-\alpha)\Pm(\bvrm)}\Big|_0^\tau
.
\end{align*}

\noindent
{\it Step 6.} Summing the above expressions 
and adding to both sides $-\inttauO{\lel \Nu_{\tau ,x};\vS(\mathbb{D}{\vw}):(\mathbb{D}_{\vv}-\mathbb{D}{\vw})\ril}$, we get all together:
\eq{      \label{mv sumstep1b}
&\bar{\cal E}_{m,v}(\Nu|\, \beta ,\BR,\BQ,\vw)\Big|_0^\tau+\inttauO{\lel\Nu_{t,x};\vS(\mathbb{D}_{\vv}-\mathbb{D}{\vw}):(\mathbb{D}_{\vv}-\mathbb{D}{\vw})\ril}+\mathcal{D}(\tau)\\
&+\frac{1}{\lambda+2\mu}\inttauO{\lel \Nu_{t,x}; s(1-s)(p_{+}(\np) -p_{-}(\nm))^2\ril }\\
&\leq \inttauO{(R+Q)\vw\cdot\pt\vw}+\inttauO{\lel \Nu_{\tau ,x};(r+q)\vv\ril\cdot\Grad\vw \cdot \vw}\\
&-\inttauO{\lel \Nu_{t ,x};(r+q)\vv\ril\cdot\pt\vw} -\inttauO{\lel \Nu_{t ,x};(r+q)\vv\otimes\vv\ril:\Grad\vw} \\
&
+ \inttauO{\lel \Nu_{\tau ,x};\vS(\mathbb{D}{\vw}):(\mathbb{D}{\vw}-\mathbb{D}_{\vv})\ril}-
\inttauO{\lel \Nu_{t,x};p(s,r,q)\ril\Div\vw}\\
%
&-\inttauO{ \left[ R  \partial_{t}\Hp'(\bvrp) + \lel \Nu_{t,x};  r\vv\ril \cdot \Grad \Hp'(\bvrp)\right] }\\
&-\inttauO{ \left[ Q  \partial_{t}\Hm'(\bvrm) + \lel \Nu_{t,x};  q\vv\ril \cdot \Grad \Hm'(\bvrm)\right] }\\
&+\intO{\lr{\alpha\Pp(\bvrp)+(1-\alpha)\Pm(\bvrm)}}\Big|_0^\tau\\
&-\lel r^M;\Grad \vw\ril_{\mathcal{M}([0,\tau]\times \overline{\Omega}), C([0,\tau]\times \overline{\Omega})}\\
&=I_1+\ldots I_{10}.
}
Note that, using \eqref{mv eqn 3}, we can  write
\eq{\label{I9}
	I_9=&\intO{\lr{\alpha\Pp(\bvrp)+(1-\alpha)\Pm(\bvrm)}}\Big|_0^\tau\\
				 =& \inttauO{ \left[ \alpha  \partial_{t}\Pp(\bvrp) + \lel \Nu_{t,x};  s\vv\ril \cdot \Grad \Pp(\bvrp) +\lel \Nu_{t,x};  s\,\tr\mathbb{D}_\vv\ril \Pp(\bvrp)\right]}\\
                &\quad+\inttauO{\lel \Nu_{t,x}; \frac{s(1-s)}{\lambda+2\mu}(p_{+}(r/s)-p_{-}(q/(1-s)))\ril \Pp(\bvrp)}\\
                &+ \inttauO{ \left[ (1-\alpha)  \partial_{t}\Pm(\bvrm) + \lel \Nu_{t,x};  (1-s)\vv\ril \cdot \Grad \Pm(\bvrm) +\lel \Nu_{t,x};  (1-s)\tr\mathbb{D}_\vv\ril \Pm(\bvrm)\right]}\\
                &\quad-\inttauO{\lel \Nu_{t,x}; \frac{s(1-s)}{\lambda+2\mu}(p_{+}(r/s)-p_{-}(q/(1-s)))\ril \Pm(\bvrm)}.
}
Using repeatedly the relation $\vr_\pm H_\pm''(\vr_\pm)=p_\pm'(\vr_\pm)$, we  first get
\eqh{
\Grad \Hp'(\bvrp)=\frac{\Grad\Pp(\bvrp)}{\bvrp},\quad \Grad \Hm'(\bvrm)=\frac{\Grad\Pm(\bvrm)}{\bvrm},
}
and similarly
\eqh{
\pt\Hp'(\bvrp)=\frac{\pt\Pp(\bvrp)}{\bvrp},\quad \pt \Hm'(\bvrm)=\frac{\pt\Pm(\bvrm)}{\bvrm},
}
\eq{      \label{mv sum1}
&\bar{\cal E}_{m,v}(\Nu|\, \beta ,\BR,\BQ,\vw)\Big|_0^\tau+\inttauO{\lel\Nu_{t,x};\vS(\mathbb{D}_{\vv}-\mathbb{D}{\vw}):(\mathbb{D}_{\vv}-\mathbb{D}{\vw})\ril}+\mathcal{D}(\tau)\\
&+\frac{1}{\lambda+2\mu}\inttauO{\lel \Nu_{t,x}; s(1-s)(p_{+}(\np) -p_{-}(\nm))^2\ril }\\
&\leq 
\inttauO{\lel \Nu_{t ,x};(r+q)(\vw-\vv)\ril\cdot\pt\vw} +\inttauO{\lel \Nu_{t ,x};(r+q)(\vw-\vv)\otimes\vv\ril:\Grad\vw} \\
&
+ \inttauO{\vS(\mathbb{D}\vw):(\mathbb{D}\vw-\mathbb{D}\vu)}-
\inttauO{\lel \Nu_{t,x};p(s,r,q)\ril\Div\vw}\\
%
&+\inttauO{ \left[\lr{\alpha-\frac{R}{\bvrp}}  \Pp'(\bvrp)\partial_{t}\bvrp + \lel \Nu_{t,x};  \lr{s\vv\bvrp-r\vv}\ril \cdot \frac{\Pp'(\bvrp)}{\bvrp}\Grad\bvrp\right] }\\
&+\inttauO{ \left[\lr{(1-\alpha)-\frac{Q}{\bvrm}}  \Pm'(\bvrm)\partial_{t}\bvrm + \lel \Nu_{t,x};  \lr{(1-s)\vv\bvrm-q\vv}\ril \cdot \frac{\Pm'(\bvrm)}{\bvrm}\Grad\bvrm\right] }\\
&+ \inttauO{ \left[\lel \Nu_{t,x};  s\,\tr\mathbb{D}_\vv\ril \Pp(\bvrp)+\lel \Nu_{t,x};  (1-s)\tr\mathbb{D}_\vv\ril \Pm(\bvrm)\right]}\\
                &+\inttauO{\lel \Nu_{t,x}; \frac{s(1-s)}{\lambda+2\mu}(p_{+}(r/s)-p_{-}(q/(1-s)))\ril \Pp(\bvrp)}\\
                &-\inttauO{\lel \Nu_{t,x}; \frac{s(1-s)}{\lambda+2\mu}(p_{+}(r/s)-p_{-}(q/(1-s)))\ril \Pm(\bvrm)}\\
                &-\lel r^M;\Grad \vw\ril_{\mathcal{M}([0,\tau]\times \overline{\Omega}), C([0,\tau]\times \overline{\Omega})}.
}

Note that in the current notation we have
\begin{align*}
   & \inttauO{\lel\Nu_{t,x};\vS(\mathbb{D}_{\vv}-\mathbb{D}\vw): \lr{\mathbb{D}_{\vv}- \mathbb{D}\vw  }\ril} \\ 
   & = 2 \mu \inttauO{\lel\Nu_{t,x};|\mathbb{T}(\mathbb{D}_{\vv})-\mathbb{T}(\mathbb{D}\vw)|^2 \ril} + \xi \inttauO{\lel\Nu_{t,x};| \tr\mathbb{D}_{\vv} - \Div \vw |^2 \ril}
\end{align*}
and $ \lambda+\frac23\mu = \xi $.

\subsection{Reduction for strong solution}
We now want to narrow down the set of the test functions. From now on, $(\beta,\BR,\BQ,\vw)$ will solve the equations of system \eqref{S} pointwisely, i.e., for $\widetilde{R}= \beta \widetilde{\varrho_+}$ and $\widetilde{Q}= (1-\beta) \widetilde{\varrho_-}$ we have
\begin{subequations}\label{reg_spec}
\begin{align}
&\ptb{\beta\bvrp}
+
\Div(\beta\bvrp\vw)=0,\label{reg_spec_1}\\
&\ptb{(1-\beta)\bvrm}
+
\Div((1-\beta)\bvrm\vw)=0,\label{reg_spec_2}\\
&\ptb{(\beta\bvrp
+(1-\beta)\bvrm)\vw}
+\Div((\beta\bvrp+(1-\beta)\vrm)\vw\otimes\vw)\nonumber\\
&\qquad\qquad\qquad\qquad\qquad
+
\Grad (\beta p_{+}(\bvrp) + (1-\beta)p_{-}(\bvrm))=\Div \vS(\Grad\vw),\label{reg_spec_3}\\
&\pt\beta+\vw\cdot\Grad\beta=\frac{\beta(1-\beta)}{\lambda+2\mu}(p_{+}(\bvrp)-p_{-}(\bvrm)).\label{reg_spec_4}
\end{align}
\end{subequations}

We will prove the following result
\begin{lemma}\label{lem:mv-relative-ineq}
Let $(\Nu,{\cal D})$ be a measure--valued solution in the sense of Definition \ref{MVdef}.  Let $(\beta,\bvrp,\bvrm,\vw)$ be a strong solution on $[0,T]$ such that
\[
        0<\underline \beta\leq \beta\leq \overline \beta<1,
        \qquad
        0<\underline\rho\leq \bvrp,\bvrm\leq \overline\rho,
\]

and with the regularity stated in Remark \ref{r:stronger_reg}.  Then, for a.a. $\tau\in(0,T)$,
\eq{\label{mv_sum_fin}
&\bar{\cal E}_{mv}\Big|_0^\tau+{\cal D}(\tau)
+\inttauO{\lel\Nu_{t,x};\vS(\mathbb{D}_{\vv}-\mathbb{D}\vw): \lr{\mathbb{D}_{\vv}- \mathbb{D}\vw  }\ril}\\
&\quad+\frac{1}{\lambda+2\mu}
\inttauO{\lel\Nu_{t,x};s(1-s)
\Big(\Pp(\np)-\Pm(\nm)-\Pp(\bvrp)+\Pm(\bvrm)\Big)^2\ril}\\
&\le{\cal R}_{mv}(\tau),
}
where the remainder is given by 
\eq{\label{mv_remainder}
{\cal R}_{mv}=&
\inttauO{\lel \Nu_{t ,x};(r+q-\BR-\BQ)(\vw-\vv)\ril\cdot\pt\vw}\\
&+\inttauO{\lel \Nu_{t ,x};((r+q)\vv-(\BR+\BQ)\vw)\cdot\Grad\vw\cdot(\vw-\vv)\ril} \\
&-
\inttauO{\lel \Nu_{t,x};s\lr{\Pp(\np)-\Pp'(\bvrp)(\np-\bvrp)-\Pp(\bvrp)}\ril\Div\vw}\\
&-
\inttauO{\lel \Nu_{t,x};(1-s)\lr{\Pm(\nm)-\Pm'(\bvrm)(\nm-\bvrm)-\Pm(\bvrm)}\ril\Div\vw}\\
&+\inttauO{ \lel \Nu_{t,x};  \lr{s\bvrp-r}(\vv-\vw)\ril \cdot \frac{\Pp'(\bvrp)}{\bvrp}\Grad\bvrp }\\
&+\inttauO{ \lel \Nu_{t,x};  \lr{(1-s)\bvrm-q}(\vv-\vw)\ril \cdot \frac{\Pm'(\bvrm)}{\bvrm}\Grad\bvrm }\\
&+ \inttauO{ \lel \Nu_{t,x};  (s-\beta)(\tr\mathbb{D}_\vv-\Div\vw)\ril \lr{\Pp(\bvrp)-\Pm(\bvrm)}}\\
&+\inttauO{\lel\Nu_{t,x};
\frac{s(1-s)(\Pp(\bvrp)-\Pm(\bvrm))}{\lambda+2\mu}
\Big[\Pp(\bvrp)-\Pp'(\bvrp)(\bvrp-\np)-\Pp(\np)\Big]\ril}\\
&-\inttauO{\lel\Nu_{t,x};
\frac{s(1-s)(\Pp(\bvrp)-\Pm(\bvrm))}{\lambda+2\mu}
\Big[\Pm(\bvrm)-\Pm'(\bvrm)(\bvrm-\nm)-\Pm(\nm)\Big]\ril}\\
&+\inttauO{\lel\Nu_{t,x};
\frac{s(\bvrp-\np)\Pp'(\bvrp)(\beta-s)(\Pp(\bvrp)-\Pm(\bvrm))}{\lambda+2\mu}\ril}\\
&+\inttauO{\lel\Nu_{t,x};
\frac{(1-s)(\bvrm-\nm)\Pm'(\bvrm)(\beta-s)(\Pp(\bvrp)-\Pm(\bvrm))}{\lambda+2\mu}\ril}\\
&-\lel r^M;\Grad\vw\ril_{\mathcal M([0,\tau]\times\overline{\Omega}),C([0,\tau]\times\overline{\Omega})} \\
=&J_1+\ldots J_{12}.
}
\end{lemma}

\pf First, it follows from \eqref{reg_spec} that $\bvrp,\bvrm$ satisfy the following equations
\eq{
&\pt{\bvrp}
+
\Div(\bvrp\vw)=-\bvrp\frac{(1-\beta)}{\lambda+2\mu}(p_{+}(\bvrp)-p_{-}(\bvrm)),\\
&\pt{\bvrm}
+
\Div(\bvrm\vw)=\bvrm\frac{\beta}{\lambda+2\mu}(p_{+}(\bvrp)-p_{-}(\bvrm)).
}
Plugging this into \eqref{mv sum1}, we obtain (recall the notation $n_+ = \frac{r}{s}$ and $n_- = \frac{q}{1-s}$)  
\eq{      \label{mv sum2}
&\bar{\cal E}_{m,v}(\Nu|\, \beta ,\BR,\BQ,\vw)\Big|_0^\tau+\inttauO{\lel\Nu_{t,x};\vS(\mathbb{D}_{\vv}-\mathbb{D}{\vw}):(\mathbb{D}_{\vv}-\mathbb{D}{\vw})\ril}+\mathcal{D}(\tau)\\
&+\frac{1}{\lambda+2\mu}\inttauO{\lel \Nu_{t,x}; s(1-s)(p_{+}(\np) -p_{-}(\nm))^2\ril }\\
&\leq \inttauO{\lel \Nu_{t ,x};(r+q)(\vw-\vv)\ril\cdot\pt\vw} +\inttauO{\lel \Nu_{t ,x};(r+q)(\vw-\vv)\otimes\vv\ril:\Grad\vw} \\
&
+ \inttauO{\vS(\DD\vw):(\DD\vw-\DD\vu)}\\
&-
\inttauO{\lel \Nu_{t,x};s\lr{\Pp(\np)-\Pp'(\bvrp)(\np-\bvrp)-\Pp(\bvrp)}\ril\Div\vw}\\
&-
\inttauO{\lel \Nu_{t,x};(1-s)\lr{\Pm(\nm)-\Pm'(\bvrm)(\nm-\bvrm)-\Pm(\bvrm)}\ril\Div\vw}\\
%
&+\inttauO{ \lel \Nu_{t,x};  \lr{s\bvrp-r}(\vv-\vw)\ril \cdot \frac{\Pp'(\bvrp)}{\bvrp}\Grad\bvrp }\\
&+\inttauO{ \lel \Nu_{t,x};  \lr{(1-s)\bvrm-q}(\vv-\vw)\ril \cdot \frac{\Pm'(\bvrm)}{\bvrm}\Grad\bvrm }\\
&+ \inttauO{ \lel \Nu_{t,x};  s(\tr\mathbb{D}_\vv-\Div\vw)\ril \Pp(\bvrp)+\lel \Nu_{t,x};  (1-s)(\tr\mathbb{D}_\vv-\Div\vw)\ril \Pm(\bvrm)}\\
&-\inttauO{ \left[\lr{\alpha\bvrp-R}  \Pp'(\bvrp)\frac{(1-\beta)}{\lambda+2\mu}(p_{+}(\bvrp)-p_{-}(\bvrm)) \right] }\\
&+\inttauO{ \left[\lr{(1-\alpha)\bvrm-Q}  \Pm'(\bvrm)\frac{\beta}{\lambda+2\mu}(p_{+}(\bvrp)-p_{-}(\bvrm))\right] }\\
                &+\inttauO{\lel \Nu_{t,x}; \frac{s(1-s)}{\lambda+2\mu}(p_{+}(\np)-p_{-}(\nm))\ril (\Pp(\bvrp)-\Pm(\bvrm))}\\
                &-\lel r^M;\Grad \vw\ril_{\mathcal{M}([0,\tau]\times \overline{\Omega}), C([0,\tau]\times \overline{\Omega})}\\
 &:=I_1+\ldots I_{12}.
}

Next, we use that the regular velocity vector field satisfies
with the velocity $\vw$ satisfying
\eqh{(\BR+\BQ)\pt{\vw}+(\BR+\BQ)\vw\cdot\Grad\vw-\Div \vS(\Grad\vw) +\Grad(\beta \Pp(\bvrp)+(1-\beta)\Pm(\bvrm)) =\vc{0}.}
Multiplying this equation by $\lel \Nu_{t ,x};\vw-\vv\ril$ and subtracting from $I_1+I_2+I_3$ above, we obtain
\eqh{
I_1+I_2+I_3=&\inttauO{\lel \Nu_{t ,x};(r+q-\BR-\BQ)(\vw-\vv)\ril\cdot\pt\vw} \\
&+\inttauO{\lel \Nu_{t ,x};((r+q)\vv-(\BR+\BQ)\vw)\cdot\Grad\vw\cdot(\vw-\vv)\ril} \\
&+\inttauO{\lel \Nu_{t,x};(\Div\vw-\tr\mathbb{D}_\vv)(\beta\Pp(\bvrp)+(1-\beta)\Pm(\bvrm))\ril},
}
where we also used the compatibility condition \eqref{mv-gradient-compatibility}.
So, all together \eqref{mv sum2} can be written as

\eq{      \label{mv sum3}
&\bar{\cal E}_{m,v}(\Nu|\, \beta ,\BR,\BQ,\vw)\Big|_0^\tau+\inttauO{\lel\Nu_{t,x};\vS(\mathbb{D}_{\vv}-\mathbb{D}{\vw}):(\mathbb{D}_{\vv}-\mathbb{D}{\vw})\ril}+\mathcal{D}(\tau)\\
&+\frac{1}{\lambda+2\mu}\inttauO{\lel \Nu_{t,x}; s(1-s)(p_{+}(\np) -p_{-}(\nm))^2\ril }\\
&\leq \inttauO{\lel \Nu_{t ,x};(r+q-\BR-\BQ)(\vw-\vv)\ril\cdot\pt\vw} \\
&+\inttauO{\lel \Nu_{t ,x};((r+q)\vv-(\BR+\BQ)\vw)\cdot\Grad\vw\cdot(\vw-\vv)\ril} \\
&-
\inttauO{\lel \Nu_{t,x};s\lr{\Pp(\np)-\Pp'(\bvrp)(\np-\bvrp)-\Pp(\bvrp)}\ril\Div\vw}\\
&-
\inttauO{\lel \Nu_{t,x};(1-s)\lr{\Pm(\nm)-\Pm'(\bvrm)(\nm-\bvrm)-\Pm(\bvrm)}\ril\Div\vw}\\
%
&+\inttauO{ \lel \Nu_{t,x};  \lr{s\bvrp-r}(\vv-\vw)\ril \cdot \frac{\Pp'(\bvrp)}{\bvrp}\Grad\bvrp }\\
&+\inttauO{ \lel \Nu_{t,x};  \lr{(1-s)\bvrm-q}(\vv-\vw)\ril \cdot \frac{\Pm'(\bvrm)}{\bvrm}\Grad\bvrm }\\
&+ \inttauO{ \lel \Nu_{t,x};  (s-\beta)(\tr\mathbb{D}_\vv-\Div\vw)\ril \lr{\Pp(\bvrp)-\Pm(\bvrm)}}\\
&-\inttauO{ \left[\lr{\alpha\bvrp-R}  \Pp'(\bvrp)\frac{(1-\beta)}{\lambda+2\mu}(p_{+}(\bvrp)-p_{-}(\bvrm)) \right] }\\
&+\inttauO{ \left[\lr{(1-\alpha)\bvrm-Q}  \Pm'(\bvrm)\frac{\beta}{\lambda+2\mu}(p_{+}(\bvrp)-p_{-}(\bvrm))\right] }\\
                &+\inttauO{\lel \Nu_{t,x}; \frac{s(1-s)}{\lambda+2\mu}(p_{+}(\np)-p_{-}(\nm))\ril [\Pp(\bvrp)-\Pm(\bvrm)]}\\
                &-\lel r^M;\Grad \vw\ril_{\mathcal{M}([0,\tau]\times \overline{\Omega}), C([0,\tau]\times \overline{\Omega})}.
}

Adding to both sides
\eqh{-2\inttauO{\lel \Nu_{t,x}; \frac{s(1-s)}{\lambda+2\mu}(p_{+}(\np)-p_{-}(\nm))\ril [\Pp(\bvrp)-\Pm(\bvrm)]}\\
+\inttauO{\lel \Nu_{t,x}; \frac{s(1-s)}{\lambda+2\mu}\ril [\Pp(\bvrp)-\Pm(\bvrm)]^2},
}
we eventually obtain

\eq{      \label{mv sum11}
&\bar{\cal E}_{m,v}(\Nu|\, \beta ,\BR,\BQ,\vw)\Big|_0^\tau+\inttauO{\lel\Nu_{t,x};\vS(\mathbb{D}_{\vv}-\mathbb{D}{\vw}):(\mathbb{D}_{\vv}-\mathbb{D}{\vw})\ril}+\mathcal{D}(\tau)\\
&+\frac{1}{\lambda+2\mu}\inttauO{\lel \Nu_{t,x}; s(1-s)(p_{+}(\np)-\Pp(\bvrp) -(p_{-}(\nm)-\Pm(\bvrm)))^2\ril }\\
&\leq \inttauO{\lel \Nu_{t ,x};(r+q-\BR-\BQ)(\vw-\vv)\ril\cdot\pt\vw} \\
&+\inttauO{\lel \Nu_{t ,x};((r+q)\vv-(\BR+\BQ)\vw)\cdot\Grad\vw\cdot(\vw-\vv)\ril} \\
&-
\inttauO{\lel \Nu_{t,x};s\lr{\Pp(\np)-\Pp'(\bvrp)(\np-\bvrp)-\Pp(\bvrp)}\ril\Div\vw}\\
&-
\inttauO{\lel \Nu_{t,x};(1-s)\lr{\Pm(\nm)-\Pm'(\bvrm)(\nm-\bvrm)-\Pm(\bvrm)}\ril\Div\vw}\\
%
&+\inttauO{ \lel \Nu_{t,x};  \lr{s\bvrp-r}(\vv-\vw)\ril \cdot \frac{\Pp'(\bvrp)}{\bvrp}\Grad\bvrp }\\
&+\inttauO{ \lel \Nu_{t,x};  \lr{(1-s)\bvrm-q}(\vv-\vw)\ril \cdot \frac{\Pm'(\bvrm)}{\bvrm}\Grad\bvrm }\\
&+ \inttauO{ \lel \Nu_{t,x};  (s-\beta)(\tr\mathbb{D}_\vv-\Div\vw)\ril \lr{\Pp(\bvrp)-\Pm(\bvrm)}}\\
&-\inttauO{ \left[\lr{\alpha\bvrp-R}  \Pp'(\bvrp)\frac{(1-\beta)}{\lambda+2\mu}(p_{+}(\bvrp)-p_{-}(\bvrm)) \right] }\\
&+\inttauO{ \left[\lr{(1-\alpha)\bvrm-Q}  \Pm'(\bvrm)\frac{\beta}{\lambda+2\mu}(p_{+}(\bvrp)-p_{-}(\bvrm))\right] }\\
                &-\inttauO{\lel \Nu_{t,x}; \frac{s(1-s)}{\lambda+2\mu}(p_{+}(\np)-p_{-}(\nm))\ril [\Pp(\bvrp)-\Pm(\bvrm)]}\\
&+\inttauO{\lel \Nu_{t,x}; \frac{s(1-s)}{\lambda+2\mu}\ril [\Pp(\bvrp)-\Pm(\bvrm)]^2}\\
&-\lel r^M;\Grad \vw\ril_{\mathcal{M}([0,\tau]\times \overline{\Omega}), C([0,\tau]\times \overline{\Omega})}\\
&=I_1+\ldots I_{12}.
}
Note that we can combine terms $I_8,\ldots,I_{11}$ to obtain
\eqh{I_8&+I_9+I_{10}+I_{11}\\
=&\inttauO{\lel\Nu_{t,x};
\frac{s(1-s)(\Pp(\bvrp)-\Pm(\bvrm))}{\lambda+2\mu}
\Big[\Pp(\bvrp)-\Pp'(\bvrp)(\bvrp-\np)-\Pp(\np)\Big]\ril}\\
&-\inttauO{\lel\Nu_{t,x};
\frac{s(1-s)(\Pp(\bvrp)-\Pm(\bvrm))}{\lambda+2\mu}
\Big[\Pm(\bvrm)-\Pm'(\bvrm)(\bvrm-\nm)-\Pm(\nm)\Big]\ril}\\
&+\inttauO{\lel\Nu_{t,x};
\frac{s(\bvrp-\np)\Pp'(\bvrp)(\beta-s)(\Pp(\bvrp)-\Pm(\bvrm))}{\lambda+2\mu}\ril}\\
&+\inttauO{\lel\Nu_{t,x};
\frac{(1-s)(\bvrm-\nm)\Pm'(\bvrm)(\beta-s)(\Pp(\bvrp)-\Pm(\bvrm))}{\lambda+2\mu}\ril},
}
which concludes the proof of Lemma \ref{lem:mv-relative-ineq}. $\Box$

\subsection{Derivation of equation for $(\alpha-\beta)^2$}
In this subsection, we derive a formula for part of the relative entropy corresponding to $\displaystyle \left[\int_\Omega \lel \Nu_{s,x}; (s-\beta)^2 \ril \,\dx \right]_{0}^{\tau}$. We prove the following result.
\begin{lemma}\label{lemm:alpha0}
Let $(R,Q,\alpha,\vu)$, $(\BR,\BQ,\beta,\vw)$ satisfying 
\eqh{
&\alpha,\beta\in L^\infty((0,T)\times\Omega)\cap C([0,T];L^1(\Omega)),\quad \vu,\vw\in L^2(0,T; W^{1,2}_0(\Omega;\R^3))}
and $R,\BR,Q,\BQ$ be such that
$$
p_+\Big(\frac{R}{\alpha}\Big)- p_-\Big(\frac{Q}{1-\alpha}\Big) \quad\text{ and}\quad  p_+\Big(\frac{\BR}{\beta}\Big)- p_-\Big(\frac{\BQ}{1-\beta}\Big) \in L^1(Q_T)
$$
be two solutions to the volume-fraction equation
\eqref{Salpha}: measure-valued and distributional, respectively. Suppose, in addition that
\begin{align*}
\Grad \beta\in L^\infty((0,T)\times\Omega;\R^3),\quad \Div\vw\in L^1(0,T; L^\infty(\Omega)).
\end{align*}
Then we have
\begin{align}\label{mv_relative_identity}
	\int_\Omega &\lel\Nu_{\tau,x}; (s-\beta(\tau,x))^2\ril\,\dx - \int_\Omega \lel\Nu_{0,x}; (s-\beta(0,x))^2\ril\,\dx \nonumber\\
	&= \inttauO{ \lel\Nu_{t,x}; (s-\beta)^2 \tr\mathbb{D}_{\vv}\ril } \nonumber\\
	&\quad - 2\inttauO{ \lel\Nu_{t,x}; (s-\beta)(\vv-\vw)\ril \cdot \Grad\beta } \nonumber\\
	&\quad + 2\inttauO{ \lel\Nu_{t,x}; (s-\beta)(\omega_s - \omega_\beta)\ril }.
\end{align}
\end{lemma}
\pf 
First observe that by definition of the measure-valued formulation, setting $\mathbf{\Phi} = \psi \mathbb{I}$  in \eqref{mv-gradient-compatibility} yields the integration-by-parts formula:
\begin{align}\label{ibp_phase}
	\inttauO{ \lel\Nu_{t,x}; \vv\ril \cdot \Grad\psi } = -\inttauO{ \lel\Nu_{t,x}; \tr\mathbb{D}_{\vv}\ril \psi }.
\end{align}

For brevity we denote the r.h.s. of the  volume-fraction equations by:
\begin{align}
	\omega_s &:= \frac{s(1-s)}{\lambda+2\mu} \left( p_+\left(\frac{r}{s}\right) - p_-\left(\frac{q}{1-s}\right) \right), \label{def:omega_s} \\
	\omega_\beta &:= \frac{\beta(1-\beta)}{\lambda+2\mu} \left( p_+\left(\frac{\BR}{\beta}\right) - p_-\left(\frac{\BQ}{1-\beta}\right) \right). \label{def:omega_beta}
\end{align}


Setting $b(s) = s^2$ ($b'(s) = 2s$) and $\psi = 1$ in \eqref{mv eqn 3}, we obtain:
\begin{align}\label{step1_identity}
	\int_\Omega \lel\Nu_{\tau,x}; s^2\ril\,\dx - \int_\Omega \lel\Nu_{0,x}; s^2\ril\,\dx
	= \inttauO{ \lel\Nu_{t,x}; s^2 \tr\mathbb{D}_{\vv} \ril } + 2\inttauO{ \lel\Nu_{t,x}; s\omega_s\ril }.
\end{align}

Next, choosing $b(s) = s$ ($b'(s) = 1$) and testing against the strong solution $\psi = \beta(t,x)$ in \eqref{mv eqn 3}, we get that:
\begin{align}\label{step2_raw}
	\int_\Omega \lel\Nu_{\tau,x}; s\ril \beta(\tau)\,\dx - \int_\Omega \lel\Nu_{0,x}; s\ril \beta(0)\,\dx
	&= \inttauO{ \lel\Nu_{t,x}; s\ril \pt\beta + \lel\Nu_{t,x}; s\vv\ril \cdot \Grad\beta } \nonumber\\
	&\quad + \inttauO{ \beta \lel\Nu_{t,x}; s\,\tr\mathbb{D}_{\vv}\ril \dx \dt+ \beta \lel\Nu_{t,x}; \omega_s\ril }.
\end{align}
Using the equation $\pt\beta = -\vw \cdot \Grad\beta + \omega_\beta$ (satisfied everywhere in $(0,T)\times \Omega$) gives:
\begin{align}\label{step2_identity}
	\int_\Omega \lel\Nu_{\tau,x}; s\ril \beta(\tau)\,\dx - \int_\Omega \lel\Nu_{0,x}; s\ril \beta(0)\,\dx
	&= \inttauO{ \lel\Nu_{t,x}; s(\vv - \vw)\ril \cdot \Grad\beta + \lel\Nu_{t,x}; s\ril \omega_\beta } \nonumber\\
	&\quad + \inttauO{ \beta \lel\Nu_{t,x}; s\,\tr\mathbb{D}_{\vv}\ril + \beta \lel\Nu_{t,x}; \omega_s\ril }.
\end{align}

Using again the equation for $\beta$ multiplied  by $2\beta$, and integrating by parts yields
\begin{align}\label{step3_identity}
	\int_\Omega \beta^2(\tau)\,\dx - \int_\Omega \beta^2(0)\,\dx = \inttauO{ \beta^2 \Div\vw } + 2\inttauO{ \beta \omega_\beta }.
\end{align}

We now assemble our final formula. Taking $\eqref{step1_identity} - 2\times\eqref{step2_identity} + \eqref{step3_identity}$, we obtain:
\begin{align}\label{recomb_raw}
	\int_\Omega &\lel\Nu_{\tau,x}; (s-\beta(\tau,x))^2\ril\,\dx - \int_\Omega \lel\Nu_{0,x}; (s-\beta(0,x))^2\ril\,\dx \nonumber\\
	&= I_{\text{trace, vel}} + 2\inttauO{ \lel\Nu_{t,x}; (s-\beta)(\omega_s - \omega_\beta)\ril },
\end{align}
where the kinetic/trace contribution is:
\begin{align*}
	I_{\text{trace, vel}} &= \inttauO{ \lel\Nu_{t,x}; s^2 \tr\mathbb{D}_{\vv}\ril } - 2\inttauO{ \beta \lel\Nu_{t,x}; s \,\tr\mathbb{D}_{\vv}\ril } \nonumber\\
	&\quad - 2\inttauO{ \lel\Nu_{t,x}; s(\vv-\vw)\ril \cdot \Grad\beta } + \inttauO{ \beta^2 \Div\vw }.
\end{align*}
Using the formula $(s-\beta)^2 = s^2 - 2s\beta + \beta^2$, the two first integrands can be rewritten as:
\begin{align*}
	\lel\Nu_{t,x}; s^2 \tr\mathbb{D}_{\vv}\ril - 2\beta \lel\Nu_{t,x}; s\, \tr\mathbb{D}_{\vv}\ril = \lel\Nu_{t,x}; (s-\beta)^2 \tr\mathbb{D}_{\vv}\ril - \beta^2 \lel\Nu_{t,x}; \tr\mathbb{D}_{\vv}\ril.
\end{align*}
Applying the duality relation \eqref{ibp_phase} with $\psi = \beta^2$ (noting $\Grad(\beta^2) = 2\beta\Grad\beta$), we obtain:
\begin{align*}
	-\inttauO{ \beta^2 \lel\Nu_{t,x}; \tr\mathbb{D}_{\vv}\ril } = 2\inttauO{ \beta \lel\Nu_{t,x}; \vv\ril \cdot \Grad\beta }.
\end{align*}
Furthermore, for the strong field $\vw \in W^{1,2}_0(\Omega;\R^3)$, standard spatial integration by parts gives:
\begin{align*}
	\inttauO{ \beta^2 \Div\vw } = -2\inttauO{ \beta \vw \cdot \Grad\beta }.
\end{align*}
Combining these back into $I_{\text{trace, vel}}$ gives:
\begin{align*}
	I_{\text{trace, vel}} &= \inttauO{ \lel\Nu_{t,x}; (s-\beta)^2 \tr\mathbb{D}_{\vv}\ril } + 2\inttauO{ \beta \lel\Nu_{t,x}; \vv\ril \cdot \Grad\beta } \nonumber\\
	&\quad - 2\inttauO{ \lel\Nu_{t,x}; s(\vv-\vw)\ril \cdot \Grad\beta } - 2\inttauO{ \beta \vw \cdot \Grad\beta } \nonumber\\
	&= \inttauO{ \lel\Nu_{t,x}; (s-\beta)^2 \tr\mathbb{D}_{\vv}\ril } + 2\inttauO{ \lel\Nu_{t,x}; \Big( \beta\vv - s(\vv-\vw) - \beta\vw \Big)\ril \cdot \Grad\beta }.
\end{align*}
Noting the algebraic identity $\beta\vv - s(\vv-\vw) - \beta\vw = -(s-\beta)(\vv-\vw)$, we conclude:
\begin{align*}
	I_{\text{trace, vel}} = \inttauO{ \lel\Nu_{t,x}; (s-\beta)^2 \tr\mathbb{D}_{\vv}\ril } - 2\inttauO{ \lel\Nu_{t,x}; (s-\beta)(\vv-\vw)\ril \cdot \Grad\beta }.
\end{align*}

Inserting $I_{\text{trace, vel}}$ into \eqref{recomb_raw} establishes the desired identity. $\Box$

\section{The weak--strong uniqueness principle}\label{Sec:5}
The purpose of this section is to prove the measure valed-strong uniqueness principle.

Our main theorem comparing global measure-valued solutions with the local strong solutions can be formulated as follows.

\begin{thm}\label{thm:entropy}
Let $(\Nu,{\mathcal D})$ be a measure-valued solution in the sense of Definition \ref{MVdef}, given by Theorem \ref{t:mv} and let $(\beta, \BR,\BQ,\vw)$ be the strong solution emanating from the data $(\beta_0 , \BR_0,\BQ_0,\vw_0)\in (W^{2,q}(\Omega))^3\times B^{3-\frac 2p}_{q,p}$, $p>2$, $q>3$, $\frac 2p + \frac 3q <1$, $0<\underline{\beta} \leq \beta_0 \leq \overline{\beta} <1$ in $\Omega$.

Then there exists a small positive parameter $\kappa$ and a positive constant $C_\kappa$ depending on $\kappa$ and on the strong solution (but not on the regularity of the weak solution), such that for a.a. $\tau\in(0,T)$
\begin{align*}
  &{\cal E}_{mv}(\Nu \,| \, \beta,\BR,\BQ,\vw)(\tau)+\mathcal{D}(\tau)\\
  &\qquad+(1-\kappa)\inttauO{\lel\Nu_{t,x};\vS(\mathbb{D}_{\vv}-\mathbb{D}{\vw}):(\mathbb{D}_{\vv}-\mathbb{D}{\vw})\ril}\\
&\qquad+\frac{1-\kappa}{\lambda+2\mu}\inttauO{\lel \Nu_{t,x}; s(1-s)(p_{+}(\np)-\Pp(\bvrp) -(p_{-}(\nm)-\Pm(\bvrm)))^2\ril }\\
  &\qquad\qquad\qquad
  \leq C_\kappa 
  {\cal E}_{mv}(\Nu\,| \, \beta,\BR,\BQ,\vw)(0).
\end{align*}

In particular, if the initial data coincide, i.e.
\eq{
\Nu_{0,x}=\delta_{(\beta_0 , \BR_0,\BQ_0,\vw_0)},
}
then 
\begin{align*}
\Nu_{\tau,x}=\delta_{(\beta(\tau,x),\BR(\tau,x),\BQ(\tau,x),\vw(\tau,x),\DD\vw(\tau,x))}
\end{align*}
for a.a. $\tau \in (0,T),\,x\in \Omega$.
\end{thm}

The proof of this theorem is split into two parts: in the next subsection we first estimate reminder terms from \eqref{mv_remainder} and then, in the following subsection we estimate the r.h.s of the formula \eqref{mv_relative_identity}. 

\subsection{Estimates of the reminder \eqref{mv_remainder}}
 We will prove it by estimating $J_1,\ldots, J_{12}$ from the remainder \eqref{mv_remainder}  term by term. {The first 6 terms can be estimated following the similar strategy from the author's previous paper \cite{LLPZ}.}

Following the approach from \cite{FG2W, AFN21, C2024, Mizerova} we introduce for $\epsilon >0$ a cut-off function
\eqh{\psi_\epsilon \in C^\infty_c((0,\infty)^2),\quad 0\leq \psi_\delta(n_+,n_-)\leq 1,\\
\psi_\epsilon=1 \quad \text{whenever}\ \epsilon<n_+,n_-<\frac{1}{\epsilon}.}
Any measurable function $f(s,r,q,\vv)$ can be written as $f=[f]_{ess}+[f]_{res}$, where
\eqh{
[f]_{ess}=\psi_\epsilon(n_+,n_-)f(s,r,q,\vv),\quad [f]_{res}=(1-\psi_\epsilon(n_+,n_-))f(s,r,q,\vv),
}
are referred to as the essential and residual components of $f$, respectively. 
We first notice that due to definitions \eqref{rel_entr}, \eqref{RQ_def} and \eqref{H_def}, we have
\begin{align} \label{res_ess}
 &  {\cal E}_{mv}(\Nu\,| \,\beta, \BR,\BQ,\vw)(\tau) \nonumber   \\
 & \qquad 
   \geq C \begin{cases}
        &\displaystyle \intO{\lel \Nu_{\tau,x};  \Big[
        s\lr{n_+ -\bvrp}^2  +(1-s)\lr{\n_- -\bvrm}^2   \Big]\ril
        }, \quad \n_+,n_- \in \left[\epsilon, 1/\epsilon\right],  \\[3ex] 
         & \displaystyle   \intO{ \lel \Nu_{\tau,x};
         \lr{1+s\n_+^\gp+(1-s)n_-^\gm}\ril
         } ,\quad  \text{ otherwise},
    \end{cases}
\end{align}
for some $\epsilon>0$ and $\tau\in(0,T)$.
We also define the essential and residual sets corresponding to $n_+,n_-$ as follows
\begin{align*}
\Omega_{ \rm ess}(t):=\{x\in\Omega: n_+, n_- \in \left[\epsilon, 1/\epsilon\right]\},
\quad 
\Omega_{\rm res}(t)=\Omega\backslash \Omega_{ \rm ess}(t).
\end{align*}

\medskip 

\noindent{\emph{Estimate of $J_1+J_2$}.} We split the first two terms of the remainder \eqref{mv_remainder} as follows
\begin{align*}
J_1+J_2=& \inttauO{\lel \Nu_{t ,x};(r+q-\BR-\BQ)(\vw-\vv)\ril\cdot\pt\vw} \\
&+\inttauO{\lel \Nu_{t ,x};((r+q)\vv-(\BR+\BQ)\vw)\cdot\Grad\vw\cdot(\vw-\vv)\ril} \\
=&\inttauO{\lel \Nu_{t ,x};\left[(r-\BR+q-\BQ)(\pt\vw+ \vw\cdot\Grad\vw)\right]\cdot(\vw-\vv)\ril}\\
&+\inttauO{\lel \Nu_{t ,x};(r+q)(\vv-\vw)\cdot\Grad\vw\cdot(\vw-\vv)\ril}.
\end{align*}
The second term might be estimated by $ \int_0^\tau C(t)\bar{\cal E}_{mv}(\Nu\,|\,\beta,\BR,\BQ,\vw)\,\dt$ with  $C$ depending on $\|\Grad\vw\|_{L^\infty(\Omega)}$.
To estimate the first term, we use
\eq{\label{decomp}
r-\BR=s(\np-\bvrp)-\bvrp(\beta-s),
}
and similarly for $q-\BQ$. Thus together we obtain
\begin{align*}
|r-\BR+q-\BQ|\leq C
\Big(
\sqrt{s}|\np-\bvrp|+\sqrt{1-s}|\nm-\bvrm|+|\beta-s| 
\Big),
\end{align*}
where $C>0$ depends only on $\|(\bvrp,\bvrm)\|_{L^\infty(\Omega)}$ and on $\|\alpha\|_{L^\infty(\Omega)}$.
Therefore
\begin{align*}
&\inttauO{\lel \Nu_{t ,x};\mathbf{1}_{ \rm ess}\left[(r-\BR+q-\BQ)(\pt\vw+\vw\cdot\Grad\vw)\right]\cdot(\vw-\vv)\ril}\\
&\quad \leq \kappa \inttauO{\lel\Nu_{t,x};|\vv-\vw|^2\ril} \\
&\qquad+\int_0^\tau C_\kappa(t)  \Big(\bar{\cal E}_{mv}(\Nu\,|\,\beta,\BR,\BQ,\vw)+\int_\Omega \lel \Nu_{s,x}; (s-\beta)^2 \ril \,\dx \Big)\,\dt,
\end{align*}
with $C_\kappa$ depending additionally on $\|\pt\vw+\vw\cdot \nabla\vw\|_{L^\infty(\Omega)}$.

For the residual part of the first integral, considering the small and large values of $\nm,\np$ separately, we get
\eq{\label{resid}
&\inttauO{\lel \Nu_{t ,x};\mathbf{1}_{ \rm res}\left[(r-\BR+q-\BQ)(\pt\vw+\vw\cdot\Grad\vw)\right]\cdot(\vw-\vv)\ril}\\
&\quad\leq \kappa \inttauO{\lel\Nu_{t,x};|\vv-\vw|^2\ril}  +\int_0^\tau C_\kappa(t) \bar{\cal E}_{mv}(\Nu\,|\,\beta,\BR,\BQ,\vw)\,\dt\\
&\qquad +\int_0^\tau C_\kappa(t)\lr{\intO{\lel \Nu_{t ,x};\mathbf{1}_{ \rm res} (r+q)\ril}}^{\frac{1}{2}}\lr{\intO{\lel \Nu_{t ,x};(r+q)|\vv-\vw|^2\ril}}^{\frac{1}{2}} \, \dt \\
&\quad \leq \kappa \inttauO{\lel\Nu_{t,x};|\vv-\vw|^2\ril}  +\int_0^\tau C_\kappa(t) \bar{\cal E}_{mv}(\Nu
\,|\,\beta,\BR,\BQ,\vw)\,\dt,
}
with $C_\kappa$ proportional to $ \|\pt\vw+\vw\cdot \nabla \vw\|_{L^\infty(\Omega)}^2$. Ultimately we have
\begin{align*}
J_1+J_2\leq \kappa \inttauO{\lel\Nu_{t,x};|\vv-\vw|^2\ril} +\int_0^\tau C_\kappa(t) \bar{\cal E}_{mv}(\Nu\,|\,\beta,\BR,\BQ,\vw)\,\dt,
\end{align*}
where $C_\kappa$ depends on $\|\pt\vw+\vw\cdot\Grad\vw\|_{L^\infty(\Omega)}$ and  on $\|\Grad\vw\|_{L^\infty(\Omega)}$.


\medskip 

\noindent{\emph{Estimate of $J_3$ and $J_4$}.} Again, both the two terms will be considered on the essential and residual parts of the domain. Using Taylor expansion for $J_3$, we obtain
\begin{align*}
&\inttauO{\lel \Nu_{t ,x};\mathbf{1}_{ \rm ess}s 
\Big(
\Pp(\bvrp)-\Pp'(\bvrp)(\bvrp-\np)-\Pp(\np)
\Big)
\Div\vw\ril}\\
&\leq C\inttauO{\lel \Nu_{t ,x};\mathbf{1}_{ \rm ess}s(\bvrp-\np)^2\ril}
\leq \int_0^\tau C(t) \bar{\cal E}_{mv}(\Nu\,|\,\beta,\BR,\BQ,\vw)\,\dt,
\end{align*}
where $C$ depends on $\|\alpha\|_{L^\infty(\Omega)}$ and $\|\Div\vw\|_{L^\infty(\Omega)}$. The estimate for the essential part of $J_4$ is similar. The residual part is bounded straight from \eqref{res_ess}. Summarizing, we have
\begin{align*}
J_3+J_4\leq  \int_0^\tau C(t) \bar{\cal E}_{mv}(\Nu\,|\,\beta,\BR,\BQ,\vw) \, \dt.
\end{align*}

\medskip 

\noindent{\emph{Estimate of $J_5$ and $J_6$}.} Both terms are again treated in an analogous way, so we estimate only one of them. Using  decomposition $\lr{s\bvrp-r}= -s(\np-\bvrp)$, and the residual part estimate analogous to \eqref{resid}, we verify that 
\begin{align*}
J_5&=\inttauO{ \lel \Nu_{t,x};  \lr{s\bvrp-r}(\vv-\vw)\ril \cdot \frac{\Pp'(\bvrp)}{\bvrp}\Grad\bvrp }\\
& \leq \kappa \inttauO{\lel\Nu_{t,x};|\vv-\vw|^2\ril} +\int_0^\tau C_\kappa(t) \bar{\cal E}_{mv}(\Nu\,|\,\beta,\BR,\BQ,\vw)\,\dt,
\end{align*}
where $C_\kappa$ depends on $\|\bvrp\|_{W^{1,\infty}(\Omega)}$. 
For $J_6$ we obtain the same estimate with $C_\kappa$ depending on $\|\bvrm\|_{W^{1,\infty}(\Omega)}$.


\medskip

Next we estimate the terms specific to the volume-fraction PDE closure relation, which did not appear in \cite{LLPZ}.

\noindent{\emph{Estimate of $J_7$}.} 
Integrating by parts we obtain
\eqh{
J_7=& \inttauO{ \lel \Nu_{t,x};  (s-\beta)(\tr\mathbb{D}_\vv-\Div\vw)\ril \lr{\Pp(\bvrp)-\Pm(\bvrm)}}\\
\leq&\kappa \inttauO{ \lel\Nu_{t,x};   |\tr\mathbb{D}_{\vv} - \Div\vw|^2\ril } +\int_0^\tau C_\kappa(t) \intO{\lel\Nu_{t,x}; (s-\beta)^2 \ril} \,\dt,}
with $C_\kappa$ depending on $\|\bvrp,\bvrm\|_{W^{1,\infty}(\Omega)}$.

\medskip
\noindent{\emph{Estimate of $J_8+J_9$}.} 
These terms are directly comparable to the potential part of the relative entropy. Splitting the terms into their essential and residual parts, we show, exactly like for $J_3+J_4$ that
\eqh{
&J_8+J_9=\\
&\inttauO{\lel\Nu_{t,x};
\frac{s(1-s)(\Pp(\bvrp)-\Pm(\bvrm))}{\lambda+2\mu}
\Big[\Pp(\bvrp)-\Pp'(\bvrp)(\bvrp-\np)-\Pp(\np)\Big]\ril}\\
&-\inttauO{\lel\Nu_{t,x};
\frac{s(1-s)(\Pp(\bvrp)-\Pm(\bvrm))}{\lambda+2\mu}
\Big[\Pm(\bvrm)-\Pm'(\bvrm)(\bvrm-\nm)-\Pm(\nm)\Big]\ril}\\
&\leq  \inttau{C(t)\intO{\lel\Nu_{t,x};s\Big(\Hp(\np)-\Hp'(\bvrp)(\np-\bvrp)-\Hp(\bvrp)
\Big)\ril}}\\
&\quad+\inttau{C(t)\intO{\lel\Nu_{t,x};(1-s)\Big(\Hm(\nm)-\Hm'(\bvrm)(\nm-\bvrm)-\Hm(\bvrm)
\Big)\ril}}\\
&\leq  \int_0^\tau C(t)\bar{\cal E}_{mv}(\Nu\,|\,\beta,\BR,\BQ,\vw)\,\dt,
}
where $C=C\lr{\gamma^+,\gamma^-,\mu,\lambda,\|\Pp(\bvrp),\Pm(\bvrm),\alpha\|_{L^\infty(\Omega)}}$.

\medskip
\noindent{\emph{Estimate of $J_{10}+J_{11}$}.} The further two terms can be again grouped together and estimated similarly to $J_5+J_6$, except that no small parameter is needed; we get
\eqh{
J_{10}+J_{11}=&\inttauO{\lel\Nu_{t,x};
\frac{s(\bvrp-\np)\Pp'(\bvrp)(\beta-s)(\Pp(\bvrp)-\Pm(\bvrm))}{\lambda+2\mu}\ril}\\
&+\inttauO{\lel\Nu_{t,x};
\frac{(1-s)(\bvrm-\nm)\Pm'(\bvrm)(\beta-s)(\Pp(\bvrp)-\Pm(\bvrm))}{\lambda+2\mu}\ril}\\
\leq& \int_0^\tau C(t)  \Big(\bar{\cal E}_{mv}(\Nu\,|\,\beta,\BR,\BQ,\vw)+\int_\Omega \lel \Nu_{s,x}; (s-\beta)^2 \ril \,\dx \Big)\,\dt,
}
where $C=C\lr{\gamma^+,\gamma^-,\|\Pp(\bvrp),\Pm(\bvrm),\alpha\|_{L^\infty(\Omega)}}$.
Finally, from \eqref{mom-def}, we obtain the estimate of the last term
\eqh{
J_{12}\leq \left\vert \lel r^M;\Grad \vw\ril_{\mathcal{M}([0,\tau]\times \overline{\Omega}), C([0,\tau]\times \overline{\Omega})} \right\vert \leq  C(\Vert \Grad\vw \Vert_{C([0,T]\times\overline{\Omega})})\int_0^\tau \chi(t) \mathcal{D}(t)\,  \dt.
}

By summing up the above estimates, we obtain that \eqref{mv_sum_fin} with the reminder given by \eqref{mv_remainder} reduces to

\eq{\label{rem_fin}
&\bar{\cal E}_{mv}(\Nu,{\mathcal D}\,|\,\beta,\BR,\BQ,\vw)\Big|_0^\tau+{\cal D}(\tau)
+\inttauO{\lel\Nu_{t,x};\vS(\mathbb{D}_{\vv}-\mathbb{D}\vw): \lr{\mathbb{D}_{\vv}- \mathbb{D}\vw  }\ril}\\
&\quad+\frac{1}{\lambda+2\mu}
\inttauO{\lel\Nu_{t,x};s(1-s)
\Big(\Pp(\np)-\Pm(\nm)-\Pp(\bvrp)+\Pm(\bvrm)\Big)^2\ril}\\
&\leq \kappa \inttauO{\lel\Nu_{t,x};|\vv-\vw|^2\ril}+\kappa \inttauO{ \lel\Nu_{t,x};   |\tr\mathbb{D}_{\vv} - \Div\vw|^2\ril }\\
&\qquad+C(\Vert \Grad\vw \Vert_{C([0,T]\times\overline{\Omega})})\int_0^\tau \chi(t) \mathcal{D}(t)\,  \dt \\
&\qquad+\int_0^\tau C_\kappa(t)  \Big(\bar{\cal E}_{mv}(\Nu\,|\,\beta,\BR,\BQ,\vw)+\int_\Omega \lel \Nu_{s,x}; (s-\beta)^2 \ril \,\dx \Big)\,\dt.
}

\subsection{Estimates of the r.h.s. of \eqref{mv_relative_identity}}

\begin{lemma}
    We have
\eq{\label{al-be}
			\int_\Omega &\lel\Nu_{\tau,x}; (s-\beta)^2\ril\,\dx\Big|_0^\tau \\
&\leq  C(\kappa, \nabla \beta,\vw)\inttauO{ \lel\Nu_{t,x}; (s-\beta)^2 \ril } \\
			&\quad +\kappa  \inttauO{ \lel\Nu_{t,x};   |\tr\mathbb{D}_{\vv} - \Div\vw|^2\ril }\\ 
			&\quad + \kappa \inttauO{ \lel\Nu_{t,x};   |\vv - \vw|^2\ril }  \\ 
			&\quad + \frac{\kappa}{\lambda+2\mu} \inttauO{\lel\Nu_{t,x};s(1-s)
\Big(\Pp(\np)-\Pm(\nm)-\Pp(\bvrp)+\Pm(\bvrm)\Big)^2\ril}.
}
\end{lemma}

\pf First let us rewrite  \eqref{mv_relative_identity} as
\begin{align}\label{mv_relative_identity_rewritten}
	\int_\Omega &\lel\Nu_{\tau,x}; (s-\beta(\tau,x))^2\ril\,\dx - \int_\Omega \lel\Nu_{0,x}; (s-\beta(0,x))^2\ril\,\dx \nonumber\\
	&= \inttauO{ \lel\Nu_{t,x}; (s-\beta)^2 (\tr\mathbb{D}_{\vv} - \Div\vw)\ril } \nonumber\\
	&\quad + \inttauO{ \lel\Nu_{t,x}; (s-\beta)^2 \Div\vw\ril } \nonumber\\
	&\quad - 2\inttauO{ \lel\Nu_{t,x}; (s-\beta)(\vv-\vw)\ril \cdot \Grad\beta } \nonumber\\
	&\quad + 2\inttauO{ \lel\Nu_{t,x}; (s-\beta)(\omega_s - \omega_\beta)\ril } = \sum_{i=1}^{4} I_i.
\end{align}
We now fix some $\kappa>0$ small, and estimate each of the terms on the r.h.s. of \eqref{mv_relative_identity}. 
For $I_1$, we have
	\begin{align*}
		|I_1| &\leq \kappa \inttauO{ \lel\Nu_{t,x};  |\tr\mathbb{D}_{\vv} - \Div\vw|^2\ril } +C(\kappa) \inttauO{ \lel\Nu_{t,x}; (s-\beta)^4 \ril } \ \\
		&\leq \kappa \inttauO{ \lel\Nu_{t,x};   |\tr\mathbb{D}_{\vv} - \Div\vw|^2\ril } +4 C(\kappa)\inttauO{ \lel\Nu_{t,x}; (s-\beta)^2 \ril },
	\end{align*}
    where to pass to the second line we used that for $s\in[0,1]$ as in \eqref{phase}, we have
	\[
	(s-\beta)^4\le\bigl(1+\|\beta\|_{L^\infty((0,T)\times\Omega)}\bigr)^2(s-\beta)^2\le 4\,(s-\beta)^2;
	\]
	whence $\langle\Nu_{t,x};(s-\beta)^4\rangle\le 4\,\langle\Nu_{t,x};(s-\beta)^2\rangle$.
    
The next estimate is straightforward, we have
	\begin{align*}
			|I_2| &\leq C(|| \div \vw||_{L^\infty }) \int_0^\tau  \lr{	\int_\Omega  \lel\Nu_{t,x}; (s-\beta(\tau,x))^2\ril\,\dx} \dt.
	\end{align*}

For $I_3$ we have 
	\begin{align*}
		|I_3| & \leq \kappa \inttauO{ \lel\Nu_{t,x};   |\vv - \vw|^2\ril } +  C(\kappa, \nabla \beta)\inttauO{ \lel\Nu_{t,x}; (s-\beta)^2 \ril }.
	\end{align*}
Finally, for the last term  we get 
	\begin{align*}
			|I_4| & \leq C(\kappa)\inttauO{ \lel\Nu_{t,x}; (s-\beta)^2 \ril } + \kappa \inttauO{ \lel\Nu_{t,x}; |\omega_s-\omega_\beta|^2 \ril } \\
			&\leq C(\kappa)\inttauO{ \lel\Nu_{t,x}; (s-\beta)^2 \ril } \\
			&\qquad + \frac{\kappa}{\lambda+2\mu}\inttauO{\lel\Nu_{t,x};s(1-s)
\Big(\Pp(\np)-\Pm(\nm)-\Pp(\bvrp)+\Pm(\bvrm)\Big)^2\ril}. 
	\end{align*}
Collecting all this, we obtain \eqref{al-be}. $\Box$

\subsection{Conclusion of the proof of Theorem \ref{thm:entropy}}
Recalling that
\eqh{{\cal E}_{mv}(\Nu\,|\,\beta,\BR,\BQ,\vw)=\bar{\cal E}_{mv}(\Nu\,|\,\beta,\BR,\BQ,\vw)+\intO{\lel\Nu_{t,x}; (s-\beta)^2 \ril},}
and summing up \eqref{rem_fin} with \eqref{al-be} we obtain 
\eq{\label{entr_fin}
&{\cal E}_{mv}(\Nu\,|\,\beta,\BR,\BQ,\vw)\Big|_0^\tau+{\cal D}(\tau)
+\inttauO{\lel\Nu_{t,x};\vS(\mathbb{D}_{\vv}-\mathbb{D}\vw): \lr{\mathbb{D}_{\vv}- \mathbb{D}\vw  }\ril}\\
&\quad+\frac{1-\kappa}{\lambda+2\mu}
\inttauO{\lel\Nu_{t,x};s(1-s)
\Big(\Pp(\np)-\Pm(\nm)-\Pp(\bvrp)+\Pm(\bvrm)\Big)^2\ril}\\
&\leq \kappa \inttauO{\lel\Nu_{t,x};|\vv-\vw|^2\ril}+\kappa \inttauO{ \lel\Nu_{t,x};   |\tr\mathbb{D}_{\vv} - \Div\vw|^2\ril }\\
&\qquad+C(\Vert \Grad\vw \Vert_{C([0,T]\times\overline{\Omega})})\int_0^\tau \chi(t) \mathcal{D}(t)\,  \dt +\int_0^\tau C_\kappa(t) {\cal E}_{mv}(\Nu\,|\,\beta,\BR,\BQ,\vw)\dt.
}

Set
\begin{align*}
 A_{\rm dev}&:=\inttauO{\lel\Nu_{t,x};
 |\mathbb T(\mathbb D_{\vv})-\mathbb T(\nabla\vw)|^2\ril},\\
 A_{\rm tr}&:=\inttauO{\lel\Nu_{t,x};
 |\tr\mathbb D_{\vv}-\Div\vw|^2\ril}.
\end{align*}
By \eqref{mv-poincare-korn}, the two terms on the right-hand side of
\eqref{entr_fin} satisfy
\begin{align*}
 &\kappa\inttauO{\lel\Nu_{t,x};|\vv-\vw|^2\ril}
 +\kappa A_{\rm tr}\\
 &\qquad\leq 3\kappa\inttauO{\lel\Nu_{t,x};
 |\vv-\vw|^2+|\mathbb D_{\vv}-\mathbb D\vw|^2\ril}
 \leq3\kappa C_{\rm PK}A_{\rm dev},
\end{align*}
where we used $|\tr\mathbb A|^2\leq3|\mathbb A|^2$. On the other hand, the
viscous term on the left-hand side is exactly
\begin{align*}
 2\mu A_{\rm dev}+\xi A_{\rm tr}.
\end{align*}
Since $\mu>0$ and $\xi\geq0$, we may choose $\kappa>0$ so small that
$3\kappa C_{\rm PK}<2\mu$. The right-hand side is then absorbed by the shear
dissipation, while the bulk term is nonnegative. Gr\"onwall's inequality concludes the proof of
Theorem \ref{thm:entropy}. $\Box$

\subsection*{Acknowledgment} The work of N.C. is supported by the NAWA ULAM grant BPN/SEL/2025/1/00005 /U/00001. The work of M.P. was partially supported by the Czech Science Foundation, project No. 25-16592S. The work of E.Z. was  supported by the EPSRC Early Career Fellowship no. EP/V000586/1.

\subsection*{Data Availability}
 Data sharing is not applicable to this article as no datasets were generated or analyzed
during the current study.

\subsection*{Conflicts of interest}
 All authors certify that there are no conflicts of interest for this work.

\subsection*{Publishing licence}

For the purpose of open access, the author has applied a Creative Commons Attribution (CC BY) licence to any Author Accepted Manuscript version arising from this submission.

\end{document}